\documentclass{siamltex}

\usepackage{amssymb,amstext,amsmath}
\usepackage{color}
\usepackage{algorithmicx}
\usepackage[ruled]{algorithm}
\usepackage{algpseudocode}
\usepackage{multirow,rotating}
\usepackage{hyperref}
\usepackage{enumitem}
\usepackage{graphicx}
\usepackage{epstopdf}
\usepackage{comment}

\newtheorem{remark}{Remark}[section]
\newtheorem{ass}{Assumption}[section]

\def\bc{\mathbf{c}}
\def\bd{\mathbf{d}}
\def\be{\mathbf{e}}
\def\bh{\mathbf{h}}
\def\bl{\mathbf{l}}
\def\bm{\mathbf{m}}
\def\bp{\mathbf{p}}
\def\bq{\mathbf{q}}
\def\br{\mathbf{r}}
\def\bs{\mathbf{s}}
\def\bu{\mathbf{u}}
\def\bv{\mathbf{v}}
\def\bw{\mathbf{w}}
\def\bx{\mathbf{x}}
\def\by{\mathbf{y}}
\def\bz{\mathbf{z}}
\def\bzero{\mathbf{0}}
\def\bbeta{\boldsymbol{\beta}}
\def\blambda{\boldsymbol{\lambda}}
\def\bmu{\boldsymbol{\mu}}

\def\btau{\boldsymbol{\tau}}

\def\bphi{\boldsymbol{\varphi}}
\renewcommand{\Re}{\mathbb{R}}

\DeclareMathOperator*{\argmin}{argmin}

\DeclareMathOperator*{\range}{range}

\graphicspath{{./Profiles/}}

\begin{document}

\title{A two-phase gradient method for quadratic programming
problems with a single linear constraint and bounds on the
variables\thanks{This work was partially supported by Gruppo Nazionale per il Calcolo
Scientifico - Istituto Nazionale di Alta Matematica (GNCS-INdAM).}}
\author{Daniela di Serafino\thanks{Dipartimento di Matematica e Fisica,
Universit\`a degli Studi della Campania L. Vanvitelli,  
viale A. Lincoln~5, 81100 Caserta, Italy, daniela.diserafino@unicampania.it.}
\and
Gerardo Toraldo\thanks{Dipartimento di Matematica e Applicazioni
R.~Caccioppoli, Universit\`a degli Studi di Napoli Federico II,
via Cintia, 80126 Napoli, Italy, toraldo@unina.it.}
\and
Marco Viola\thanks{Dipartimento di Ingegneria Informatica Automatica e
Gestionale A. Ruberti, Sapienza - Universit\`a di Roma,
via Ariosto 25, 00185 Roma, Italy, marco.viola@uniroma1.it.}
\and
Jesse Barlow\thanks{Department of Computer Science and Engineering, 
Pennsylvania State University, 
343G IST Building, University Park, PA 16802-6822, USA,
barlow@cse.psu.edu.}}

\maketitle

\centerline{\footnotesize FINAL~VERSION -- May 25, 2018}

\begin{abstract}
We propose a gradient-based method for quadratic programming problems with a
single linear constraint and bounds on the variables.
Inspired by the GPCG algorithm for bound-constrained convex
quadratic programming [J.J. Mor\'e and G. Toraldo, SIAM J.~Optim.~1, 1991],
our approach alternates between two phases until convergence: an
identification phase, which performs gradient projection iterations until
either a candidate active set is identified or no reasonable progress is
made, and an unconstrained minimization phase, which reduces the objective
function in a suitable space defined by the identification phase, by
applying either the conjugate gradient method or a recently proposed
spectral gradient method. However, the algorithm differs from GPCG not only
because it deals with a more general class of problems,
but mainly for the way it stops the minimization phase.
This is based on a comparison between a measure
of optimality in the reduced space and a measure of bindingness
of the variables that are on the bounds, defined by extending the concept
of proportional iterate, which was proposed by some authors for box-constrained problems.
If the objective function is bounded, the algorithm converges to a stationary point
thanks to a suitable application of the gradient projection method in the identification
phase. For strictly convex problems, the algorithm converges to the optimal solution in a finite
number of steps even in case of degeneracy. Extensive numerical
experiments show the effectiveness of the proposed approach.
\end{abstract}


	

\section{Introduction\label{sec:intro}}
	
We are concerned with the solution of Quadratic Programming problems with a Single
Linear constraint and lower and upper Bounds on the variables (SLBQPs): \\[-5mm]
\begin{equation}\label{SLBQP}
\begin{array}{rl}
\min             & \displaystyle f(\bx) := \frac{1}{2} \bx^T\,H\,\bx - \bc^T\bx , \\
\mbox{s.t.}   & \displaystyle \bq^T\bx=b, \;\, \bl\leq \bx\leq \bu,
\end{array}
\end{equation}
where $H \! \in \Re^{n \times n}$ is symmetric, $\bc, \bq \in \Re^{n}$,
$b \in \Re$, $\bl \in \! \left\lbrace\Re \cup \{ -\infty \} \right\rbrace^{n}$,
$\bu \in \! \left\lbrace\Re \cup \{ +\infty \} \right\rbrace^{n}$,
and, without loss of generality, $l_i<u_i$ for all $i$.
In general, we do not assume that the problem is strictly convex.
SLBQPs arise in many applications, such as support vector machine
training~\cite{Vapnik:1982}, portfolio selection~\cite{PardalosRosen:1987},
multicommodity network flow and logistics \cite{Meyer:1984}, and statistics estimate
from a target distribution \cite{Amaral:2017}.
Therefore, designing efficient methods for the solution of~\eqref{SLBQP}
has both a theoretical and a practical interest.

Gradient Projection (GP) methods are widely used to solve large-scale SLBQP problems,
thanks to the availability of low-cost projection algorithms onto the feasible set
of~\eqref{SLBQP} (see, e.g, \cite{CalamaiMore:1987a,Dai:2006a,Condat:2016}).
In particular, Spectral Projected Gradient methods~\cite{Birgin:2000}, other GP
methods exploiting variants of Barzilai-Borwein (BB) steps~\cite{Serafini:2005,Dai:2006a},
and more recent Scaled Gradient Projection methods \cite{Bonettini:2009}
have proved their effectiveness in several applications.

Bound-constrained Quadratic Programming problems (BQPs) can
be regarded as a special case of SLBQPs, 
where the theory or the implementation can be simplified.
%
%
This has favoured the development of more specialized gradient-based methods,
built upon the idea of combining steps aimed at identifying the variables that are active
at a solution (or at a stationary point) with unconstrained minimizations in reduced spaces
defined by fixing the variables that are estimated active
\cite{More:1989,Friedlander:1989,More:1991,Friedlander:1994b,Bielschowsky:1997,
Dostal:1997,Dostal:2005,Hager:2006,DostalPospisil:2015,Hager:2016}.
Thanks to the identification properties of the GP method~\cite{CalamaiMore:1987} and
to its capability of adding/removing multiple variables to/from the active set in a single iteration,
GP steps are a natural choice to determine the active variables. A well-known method
based on this approach is GPCG~\cite{More:1991}, developed for strictly convex BQPs.
It alternates between two phases: an identification phase, which performs GP iterations
until a suitable face of the feasible set is identified or no reasonable progress toward
the solution is achieved, and a minimization phase, which uses the Conjugate Gradient (CG)
method to find an approximate minimizer of the objective function in the reduced space
resulting from the identification phase. We note that the global convergence of the GPCG
method relies on the global convergence of GP with steplengths satisfying a suitable
sufficient decrease condition~\cite{CalamaiMore:1987}. 
Furthermore, GPCG has finite convergence under a dual nondegeracy assumption,
thanks to the ability of the GP method to identify the active constraints in a
finite number of iterations~\cite{CalamaiMore:1987}, and to the finite termination of
the CG method. Finally, the identification property also holds for
nonquadratic objective functions and polyhedral constraints,
and thus the algorithmic framework described so far can
be extended to more general problems.

Here we propose a two-phase GP method for SLBQPs,
called \emph{Proportionality-based 2-phase Gradient Projection (P2GP)} method,
inspired by the GPCG algorithm. Besides targeting problems
more general than strictly convex BQPs, the new method differs from GPCG
because it follows a different approach in deciding when to terminate optimization
in the reduced space. Whereas GPCG uses a heuristics based on the
bindingness of the active variables, P2GP relies on the comparison
between a measure of optimality within the reduced space and a measure
of bindingness of the variables that are on the bounds. This approach exploits the concept
of \emph{proportional iterate}, henceforth also refereed to as \emph{proportionality}.
This concept, presented by Dost\'al for strictly convex BQPs \cite{Dostal:1997}, is based on the splitting of the
optimality conditions between \emph{free} and \emph{chopped gradients}, firstly introduced by Friedlander
and Mart\'{i}nez in \cite{Friedlander:1989}.
To this end, we generalize the definition of free and chopped gradients
to problem~\eqref{SLBQP}. As in GPCG, and unlike other algorithms for
BQPs sharing a common ground (e.g.,
\cite{Dostal:1997,DostalPospisil:2015,Dostal:2005,Robinson:2015}),
the task of adjusting the active set is left only to the GP steps;
thus, for strictly convex BQPs our algorithm differs from GPCG in the criterion
used to stop minimization of the reduced problem. This change 
makes a significant difference in the effectiveness of the algorithm  as our numerical
experiments show. In addition, the application of the proportionality concept allows
to state finite convergence for strictly convex problems also for dual-degenerate solutions.
More generally, if the objective function is bounded, the algorithm converges to
a stationary point as a result of suitable application of the GP method
in the identification phase.

About the GP iterations, we note that the identification property holds provided
that a sufficient decrease condition holds, and therefore the choice of the Cauchy
stepsize as initial trial value in the projected gradient steps~\cite{More:1989,More:1991}
can be replaced by rules used by new spectral gradient methods.
Inspired by encouraging results reported for BQPs in~\cite{Dai:2005}
and by further studies on steplength selection in gradient
methods~\cite{diSerafino:2016,diSerafino:2017}, we consider a
monotone version of the Projected BB method which uses the
$\textrm{ABB}_{\textrm{min}}$ steplength introduced in~\cite{Frassoldati:2008}.

In the minimization phase, we use the CG method, and, in the strictly convex case,
we also use the SDC method proposed in~\cite{Deasmundis:2014}.
This provides a way to extend SDC to the costrained case, with the goal of
exploiting its smoothing and regularizing effect observed on
certain unconstrained ill-posed inverse problems~\cite{Deasmundis:2016}.
Of course, the CG solver is still the reference choice in general, especially
because it is able to deal with nonconvexity through directions of negative curvature
(as done, e.g., in \cite{Robinson:2015}), whereas handling negative curvatures
with spectral gradient methods may be a non-trivial
task (see, e.g., \cite{Curtis:2016} and the references therein).

This article is organized as follows. In Section~\ref{sec:stationarity}, we recall
stationarity results for problem~\eqref{SLBQP}. In Section~\ref{sec:proportioning},
we define free and chopped gradients for SLBQPs and show how they can be used
to extend the concept of proportionality to this class of problems. In Section~\ref{sec:method},
we describe the P2GP method and state its convergence properties. We discuss the results
of extensive numerical experiments in Section~\ref{sec:experiments}, showing the
effectiveness of our approach. We draw some conclusions in 
Section~\ref{sec:conclusions}.


\subsection{Notation\label{sec:notation}}

Throughout this paper scalars are denoted by lightface
Roman fonts, e.g., $a \in \Re$, vectors by boldface Roman fonts, e.g., $\bv \in \Re^n$,
and matrices by italicized lightface capital fonts, e.g., $M\in\Re^{n\times n}$.
The vectors of the standard basis of $ \Re^n $ are indicated as $ \be_1,\ldots,\be_n$. 
Given $\mathcal{R}, \mathcal{C} \subseteq \left\lbrace 1, \ldots, n \right\rbrace$, we set \\[-12pt]
$$
   \bv_{\mathcal{R}} := 
   (v_i)_{i\in\mathcal{R}}, \quad
   M_{\mathcal{R} \mathcal{C}} := 
   \left( m_{ij} \right)_{i\in\mathcal{R}, j\in\mathcal{C}},
$$
where $v_i$ is the $i$th entry of $\bv$ and $m_{ij}$ the $(i,j)$th entry of $M$.
For any vector $\bv$, $\left\lbrace \bv \right\rbrace^\perp$ is
the space orthogonal to $\bv$.
For any symmetric matrix $M$, we use $ \kappa(M)$, $\zeta_{min}(M)$
and $\zeta_{max}(M)$ to indicate the condition number, and the minimum
and maximum eigenvalue of $M$, respectively. Norms $\| \cdot \|$
are $\ell_2$, unless otherwise stated.

The feasible set, $\Omega$, of problem~\eqref{SLBQP} is given by
$$
   \Omega := \left\lbrace \bx \in \Re^n \, : \, \bq^T\bx = b 
     \;\wedge\; \bl\leq \bx\leq \bu\right\rbrace.
$$
For any $\bx \in \Omega$, we define the following index sets:
\begin{eqnarray*}
\begin{array}{ll}
    \displaystyle \mathcal{A}_l({\bx}) := \{i \, : \, {x}_i=l_i \},
       & \displaystyle \ \ \mathcal{A}_u({\bx}) := \{i \, : \, {x}_i=u_i \}, \\[4pt]
    \displaystyle \mathcal{A}({\bx}) := \mathcal{A}_l({\bx}) \cup \mathcal{A}_u({\bx}),
       & \displaystyle \ \ \mathcal{F}({\bx}) := \{1,\ldots,n\} \setminus \mathcal{A}({\bx}).
\end{array}
\end{eqnarray*}
$\mathcal{A}(\bx)$ and $\mathcal{F}(\bx)$ are called the active and free sets at $\bx$,
respectively. Given $\bx, \by \in \Omega$, by writing $ \mathcal{A}(\bx) \subseteq \mathcal{A}(\by)$
we mean that
$$
\mathcal{A}_l(\bx) \subseteq \mathcal{A}_l(\by), \quad
\mathcal{A}_u(\bx) \subseteq \mathcal{A}_u(\by)
$$
both hold. For any $\bx\in\Omega$, we also set
\begin{eqnarray}
   \Omega(\bx) & := & \left\lbrace \bv \in \Re^n \, : \, \bq^T\bv = b \;
   \wedge \;  v_i=x_i \,\; \forall \,i \in \mathcal{A}(\bx) \right\rbrace, \label{omegax} \\
   \Omega_0(\bx) & := & \left\lbrace \bv \in \Re^n \, : \, \bq^T\bv = 0 \;
   \wedge \; v_i=0 \,\; \forall \, i \in \mathcal{A}(\bx) \right\rbrace. \label{omega0x}
\end{eqnarray}
Note that $\Omega(\bx)$ is the affine closure of the face determined by the active set at $\bx$.

We use superscripts to denote the elements of a sequence, e.g.,
$\left\lbrace \bx^k\right\rbrace $; furthermore, in order to simplify
the notation, for any $\bx^k$ and $\bx^*$ we also define
\begin{eqnarray*}
	& & f^k := f(\bx^k),
       \quad \nabla f^k := \nabla f(\bx^k), 
       \quad \mathcal{A}^k := \mathcal{A}(\bx^k), 
       \quad \mathcal{F}^k := \mathcal{F}(\bx^k), \\
	& & f^* := f(\bx^*),
       \quad \, \nabla f^* := \nabla f(\bx^*),
       \quad \mathcal{A}^* := \mathcal{A}(\bx^*), 
       \quad \mathcal{F}^* := \mathcal{F}(\bx^*).
\end{eqnarray*}

Finally, for any finite set $\mathcal{T}$, we denote by $\vert \mathcal{T} \vert$ its cardinality.


\section{Stationarity results for SLBQPs\label{sec:stationarity}}
		
We recall that $\bx^* \in \Omega$ is a stationary point for problem~\eqref{SLBQP}
if and only if there exist Lagrange multipliers $\rho^*, \lambda_i^* \in \Re$, with
$i \in \mathcal{A}^*$, such that	\\[-14pt]
\begin{equation} \label{slbqp_kkt}
   \nabla f^* = \sum_{i \in \mathcal{A}^*} \lambda_i^* \be_i + \rho^* \bq,
   \quad \lambda_i^* \ge 0 \; \mbox{ if } \;  i \in \mathcal{A}_l^*,
   \quad \lambda_i^* \le 0 \;  \mbox{ if } \;  i \in \mathcal{A}_u^*,
\end{equation}
\vskip -4pt
\noindent or, equivalently, \\[-8mm]
\begin{eqnarray}
   && \qquad\qquad\qquad\qquad\qquad
        \nabla f^*_{\mathcal{F}^*} - \rho^* \bq_{\mathcal{F}^*} =  \bzero, 
        \label{slbqp_kkt_compwise} \\
  && \lambda^*_i=\nabla f _i^* - \rho^* q_i \geq 0   \;
       \mbox{ if }  \; i \in \mathcal{A}_l^*, \quad
        \lambda^*_i= \nabla f_i^* - \rho^* q_i \leq 0  \;
        \mbox{ if }  \; i \in \mathcal{A}_u^*.
       \label{slbqp_kkt_compwise_2} 
\end{eqnarray}
If $\bq_{\mathcal{F}^*} \! \ne \bzero$,
by taking the scalar product of~\eqref{slbqp_kkt_compwise} with $\bq_{\mathcal{F}^*}$,
we obtain 	
$$
  \rho^* =
   \frac{\bq_{\mathcal{F}^*}^T \, \nabla f^*_{\mathcal{F}^*}}	  
    {\bq_{\mathcal{F}^*}^T \, \bq_{\mathcal{F}^*}}
$$
(with a little abuse of notation we include $\mathcal{F}^* = \emptyset$
in the case $\bq_{\mathcal{F}^*} \! = \bzero$). Then, by defining for all $\bx \in \Omega$
\\[-7mm]
\begin{equation} \label{defrho}
\rho(\bx) :=
\left\{  
\begin{array}{cl}
                      \!\! 0        & \mbox{if } \; \bq_{\mathcal{F}} \! = \bzero, \\
\displaystyle  \!\! \frac{\bq_{\mathcal{F}}^T \, \nabla f _{\mathcal{F}}(\bx)}
                {\bq_{\mathcal{F}}^T \, \bq_{\mathcal{F}}} &  \mbox{otherwise},
\end{array} 
\right.
\end{equation}
where $\mathcal{F} = \mathcal{F}(\bx)$, and \\[-8mm]
\begin{equation} \label{defh}
\bh(\bx) := \nabla f (\bx) - \rho(\bx)\bq ,
\end{equation}
conditions \eqref{slbqp_kkt_compwise}-\eqref{slbqp_kkt_compwise_2}
can be expressed as
\begin{equation} \label{hstar}
   h_i^*=0 \; \mbox{ if } i \in  \mathcal{F}^*, \quad
   h_i^*\geq 0 \; \mbox{ if }  i \in  \mathcal{A}_l^*, \quad
   h_i^*\leq 0 \; \mbox{ if }  i \in  \mathcal{A}_u^*.
\end{equation}
This suggests the following definition.
\par \smallskip
\begin{definition}[Binding set]
For any $\bx \in \Omega$, the binding set at $\bx$ is defined as
\begin{equation} \label{definition_binding}
\mathcal{B}({\bx}) := 
\left\lbrace i  \, : \, \left( i \in \mathcal{A}_l (\bx) \, \wedge \, h_i (\bx) \geq 0 \right)
\; \vee \;
\left( i\in\mathcal{A}_u (\bx) \, \wedge \,  h_i (\bx) \leq  0 \right)\right\rbrace.
\end{equation}
\end{definition}

\vskip -10pt \noindent 
We note that, for the BQP case, \eqref{definition_binding} 
corresponds to the standard definition of binding set where $\bh(\bx)$
is replaced by $\nabla f (\bx)$.

We can also provide an estimate of the Lagrange multipliers based on \eqref{defh},
as stated by the following theorem.
\par \smallskip
\begin{theorem} \label{theorem_lambda}
Assume that $\left\{ \bx^k \right\}$ is a sequence in $\Omega$ that
converges to a nondegenerate stationary point $\bx^*$, and
$\mathcal{A}(\bx^k ) = \mathcal{A}(\bx^* )$ for all $k$ sufficiently large. Then
\begin{equation}
\lim_{k \rightarrow \infty} \rho(\bx^k) = \rho^* \quad \mbox{and}  \quad
\lim_{k \rightarrow \infty} \lambda_i(\bx^k) = \lambda^*_i
        \;\;\, \forall i \in \mathcal{A}^*,
\end{equation}
where $\lambda_i(\bx)$ is defined as follows:
\begin{equation*}
\lambda_i(\bx) := \left\{ \!\! \begin{array}{cl}
\max \{ 0, h_i(\bx) \}  & \!\! \mbox{if }  \; i \in \mathcal{A}_l(\bx), \\
\min \{ 0, h_i(\bx) \}  & \!\! \mbox{if }  \; i \in \mathcal{A}_u(\bx),\\
0 & \!\! \mbox{if }  \; i \in \mathcal{F}(\bx).\end{array} \right.
\end{equation*}
\end{theorem}
\begin{proof}
The result is a straightforward consequence of the continuity of $\nabla f$.
\end{proof}

\smallskip
\begin{remark}
Theorem~\ref{theorem_lambda} can be rephrased by saying that $\left( \rho(\bx),
\, \blambda(\bx)^T \right)^T$, where $\blambda(\bx) =
(\lambda_i(\bx))_{i \in \mathcal{A}(\bx)}$,
is a consistent Lagrange multiplier estimate for~\eqref{SLBQP}, according to the definition
in \cite[p.~107]{CalamaiMore:1987}.
\end{remark}


\smallskip
Another way to express stationarity for problem~\eqref{slbqp_kkt}
is by using the projected gradient of $f$ at a point $\bx \in \Omega$,
defined by Calamai and Mor\'{e}~\cite{CalamaiMore:1987} as
\begin{equation} \label{pg_grad_minprob}
\nabla_{\Omega} f(\bx) := \argmin \left\lbrace  \Vert \bv + \nabla f(\bx)\Vert 
\; \mbox{ s.t.} \; \bv \in T_\Omega(\bx) \right\rbrace,
\end{equation}
\par \vskip -12pt
\noindent where \\[-14pt]
$$
  T_\Omega(\bx) = \left\lbrace \bv \in \Re^n \, : \; \bq^T\bv=0 \; \wedge \;
  v_i \geq 0 \;\, \forall \, i \in \mathcal{A}_l(\bx) \; \wedge \;
  v_i \leq 0 \;\, \forall \, i \in \mathcal{A}_u(\bx)\right\rbrace
$$
is the tangent cone to $\Omega$ at $\bx$, i.e., the closure of the cone of all feasible
directions at~$\bx$. It is well known that $\bx^* \in \Omega$ is a stationary point
for~\eqref{SLBQP} if and only if $\nabla_{\Omega} f(\bx^*) = \bzero$,
which is equivalent to
$$
-\nabla f(\bx^*) \in T_\Omega(\bx)^\circ,
$$
where $T_\Omega(\bx)^\circ= \left\lbrace \bw \in \Re^n \, : \, \bw^T\bv
\leq 0  \;\, \forall \, \bv \in  T_\Omega(\bx) \right\rbrace $
is the polar of the tangent cone at $\bx$, i.e., the normal cone to $\Omega$ at $\bx$.

In the method proposed in this work we use the projected gradient as a measure of stationarity.
It could be argued that the projected gradient is inappropriate
to measure closeness to a stationary point, since it is only lower semicontinuous
(see \cite[Lemma~3.3]{CalamaiMore:1987}); because of that, e.g., Mohy-ud-Din
and Robinson in their algorithm prefer to use the so-called reduced free and chopped
gradients \cite{Robinson:2015}. However, Calamai and Mor\'e in \cite{CalamaiMore:1987}
show that the limit points of a bounded sequence $\{ \bx^k \}$ generated by any GP algorithm
are stationary and 
\begin{equation} \label{gptozero}
\lim_{k \rightarrow \infty} \Vert \nabla_{\Omega} f(\bx^k)\Vert = 0,
\end{equation}  
provided the steplengths are bounded and satisfy suitable sufficient decrease conditions.
Similar results hold for a more general algorithmic framework
(see~\cite[Algorithm~5.3]{CalamaiMore:1987}), which GPCG as well as the new
method P2GP fit into. Another important issue is that, for any sequence $\{ \bx^k \}$
converging to a nondegenerate stationary point $\bx^*$, if \eqref{gptozero} holds
then $\mathcal{A}^k = \mathcal{A^*}$ for all $k$ sufficiently large. However,
for problem~\eqref{SLBQP}, condition \eqref{gptozero} has an important
meaning in terms of active constraints identification even in case of degeneracy,
provided the following constraint qualification holds.
\par \smallskip
\begin{ass}[Linear Independence Constraint Qualification - LICQ] \label{licq}	
Let~$\bx^*$ be any stationary point of \eqref{SLBQP}. The active constraint normals
$\left\lbrace \bq \right\rbrace \cup \left\lbrace \be_i \, : \, i\in\mathcal{A}^* \right\rbrace$
are linearly independent.
\end{ass}
\par \smallskip \noindent
This assumption is not very restrictive; for instance, it is always satisfied if $\Omega$
is the standard simplex. Furthermore, it guarantees $\bq_{\mathcal{F}^*} \! \ne \bzero$.

The following proposition summarizes the convergence properties for a sequence $\{ \bx^k \}$
satisfying~\eqref{gptozero}, both in terms of stationarity and active set identification.
\par \smallskip
\begin{theorem} \label{th:convres_calamaimore_gptozero}
Assume that $\left\{ \bx^k \right\}$ is a sequence in $\Omega$ 
that converges to a point~$\bx^*$ and
$\lim_{k \rightarrow \infty} \Vert \nabla_{\Omega} f(\bx^k)\Vert = 0$. Then 
\begin{itemize}
\item[(i) ] $\bx^*$ is a stationary point for problem \eqref{SLBQP};
\item[(ii) ] if Assumption~\ref{licq} holds, then $\mathcal{A}_N^*\subseteq \mathcal{A}(\bx^k)$
for all $k$ sufficiently large, where $\mathcal{A}_N^* = \left\lbrace i \in \mathcal{A}^* \, : \, 
\lambda_i^* \neq 0 \right\rbrace$ and $\lambda_i$ is the Lagrange multiplier associated with
the $i$th bound constraint.
\end{itemize}
%
%
\end{theorem}
\begin{proof}
Item~(i) trivially follows from the lower semicontinuity of $\Vert \nabla_\Omega f(\bx) \Vert$.
%
%

Item (ii) extends \cite[Theorem 4.1]{CalamaiMore:1987} to degenerate stationary points
that satisfy Assumption~\ref{licq}. We first note that, since $\left\{ \bx^k \right\}$
converges to $\bx^*$, we have $\mathcal{F}^*\subseteq\mathcal{F}^k$
and hence $\mathcal{A}^k\subseteq\mathcal{A}^*$ for all $k$ sufficiently large.
The proof is by contradiction. Assume that there is an index $m$ and an infinite set 
$K\subseteq\mathbb{N}$ such that $ m \in \mathcal{A}_N^*\setminus\mathcal{A}^k$
for all $k\in K$. Without loss of generality, we assume $x^*_m = u_m$ and thus 
$\lambda^*_m<0$.
Let $P_{\Phi}$ be the orthogonal projection onto
$$
\Phi = \left\lbrace \bv \in \Re^n \, : \, \bq^T \bv=0 \; \wedge \;
\left( \be_i^T \bv = 0 \;\, \forall \, i \in \mathcal{A}^*, \, i \ne m \right) \right\rbrace .
$$
Assumption~\ref{licq} implies that ${P}_{\Phi}(\be_m)\neq 0$.
%
%
Since $m \notin\mathcal{A}(\bx^k)$,
it is $P_{\Phi}(\be_m)\in T_\Omega(\bx^k)$.
%
%
Then, by~\cite[Lemma~3.1]{CalamaiMore:1987},
%
%
$$
   \nabla f (\bx^k)^T\,P_{\Phi}(\be_m) \geq - \left\Vert
   \nabla_\Omega f(\bx^k)\right\Vert\;\left\Vert P_{\Phi}(\be_m) \right\Vert,
$$
and since $\left\lbrace \bx^k\right\rbrace$ converges to 
$\bx^*$ and $\left\{\left\Vert \nabla_\Omega f(\bx^k)\right\Vert\right\}$ converges to $0$,
we have
$$
\nabla f (\bx^*)^T\,{P}_{\Theta}(\be_m) \geq 0.
$$
On the other hand, by \eqref{slbqp_kkt} and the definition of $P_{\Phi}$ we get
$$
   \nabla f(\bx^*)^T P_{\Phi}(\be_m) = 
   \sum_{i \in \mathcal{A}^*} \lambda_i^* \be_i^T P_{\Phi}(\be_m) + 
   \theta^* \bq^T P_{\Phi}(\be_m)  =
   \lambda_m^*\be_m^T \, P_\Phi(\be_m) < 0,
$$
\par \vskip -5pt \noindent
where the last inequality derives from $\lambda^*_m<0$ and $(\be_m)^T
{P}_{\Phi}(\be_m) = \Vert {P}_{\Phi}(\be_m)\Vert^2 > 0$.
The contradiction proves that the set $K$ is finite, and hence $m \in \mathcal{A}^k$
for all $k$ sufficiently large.
\end{proof}
\par \smallskip
\noindent
By Theorem~\ref{th:convres_calamaimore_gptozero},
if an algorithm is able to drive the projected gradient toward zero, then it is able
to identify the active variables that are nondegenerate at the solution in a finite number of iterations.


\section{Proportionality\label{sec:proportioning}}

A critical issue about a two-phase method like GPCG stands
in the approximate minimization of $f(\bx)$ in the reduced space defined
according to the working set inherited from the GP iterations. This is an
unconstrained minimization phase in which the precision required should
depend on how much that space is worth to be explored. For strictly convex BQPs,
Dost\'al introduced the concept of proportional iterate~\cite{Dostal:1997,Dostal:2005},
based on the ratio between a measure of optimality within the reduced
space and a measure of optimality in the complementarity space. Similar ideas
have been discussed in \cite{Friedlander:1989,Friedlander:1994b,Bielschowsky:1997}.
According to~\cite{Dostal:1997}, $\bx^k$ is called proportional if, for a suitable
constant $\Gamma>0$,
\begin{equation} \label{proportioning}	
    \Vert \bbeta(\bx^k) \Vert_\infty\le \Gamma \Vert \bphi(\bx^k) \Vert,
\end{equation}
where $\bphi(\bx)$ and $\bbeta(\bx)$ are the so-called free and chopped gradients,
respectively, defined componentwise as
$$
   \varphi_i(\bx)  :=  \left\{ \!\! \begin{array}{cl}
	                \nabla f_i(\bx) & \mbox{if } \; i \in  \mathcal{F}(\bx), \\
			 0  & \mbox{if } \; i \in  \mathcal{A}_l(\bx), \\
			 0  & \mbox{if } \; i \in  \mathcal{A}_u(\bx),
			\end{array} 
\right. \quad
\beta_i(\bx)  :=  \left\{ \!\! \begin{array}{cl}
			 0  &\mbox{if } \; i \in  \mathcal{F}(\bx), \\
			 \min \{0,\nabla f_i(\bx) \}  & \mbox{if } \; i \in  \mathcal{A}_l(\bx), \\
			 \max \{0,\nabla f_i(\bx) \} & \mbox{if } \; i \in  \mathcal{A}_u(\bx).
			\end{array} 
\right.
$$
We note that $\bx^*$ is stationary for the BQP problem if and only if
$$
\Vert \bbeta(\bx^*)\Vert + \Vert\bphi(\bx^*)\Vert = 0;
$$
Furthermore, when the Hessian of the objective function is positive definite,
disproportionality of $\bx^k$  guarantees that the solution of the BQP problem
does not belong to the face determined by the active variables at $\bx^k$,
and thus exploration of that face is stopped.

In the remainder of this section, to measure the violation of the KKT
conditions~\eqref{slbqp_kkt_compwise}-\eqref{slbqp_kkt_compwise_2}
and to balance optimality between free and active variables, we give suitable
generalizations of the free and chopped gradient for the SLBQPs.
%
%
As in~\cite{Dostal:1997}, we exploit the free and the chopped gradient to decide when
to terminate minimization in the reduced space, and to state finite convergence for strictly
convex quadratic problems even in case of degeneracy at the solution. For simplicity,
in the sequel we adopt the same notation used for BQPs.

We start by defining the free gradient $\bphi(\bx)$ at $\bx \in \Omega$
for the SLBQP problem.
\smallskip
\begin{definition}
For any $\bx \in \Omega$, the free gradient $\bphi(\bx)$ is defined as follows:
$$
   \varphi_i(\bx) := 
   \left\lbrace \!\! \begin{array}{cl}
	h_i(\bx) & \mbox{if } \; i \in \mathcal{F}(\bx),\\
	0           & \mbox{if } \; i \in \mathcal{A}(\bx),
   \end{array}  \right.
$$
where $\bh(\bx)$ is given in~\eqref{defh}.
\end{definition}
\par \smallskip \noindent
We note that \\[-18pt]
\begin{equation}\label{phi_F as projection}
\bphi_{\mathcal{F}}(\bx) = P_{\{\bq_{\mathcal{F}} \}^{\bot}}
\left(\nabla f_{\mathcal{F}}(\bx)\right),
\end{equation}
where $\mathcal{F} = \mathcal{F}(\bx)$ and $P_{\{\bq_{\mathcal{F}} \}^{\bot}} \! \in 
\Re^{\vert\mathcal{F}\vert \times \vert\mathcal{F}\vert}$
is the orthogonal projection onto the subspace of $\Re^{\vert\mathcal{F}\vert}$
orthogonal to $\bq_\mathcal{F}$
(i.e., the nullspace of $\bq_\mathcal{F}^T$),
$$
P_{\{\bq_{\mathcal{F}} \}^{\bot}} = I -
\frac{\bq_\mathcal{F} \, \bq_\mathcal{F}^T}{{\bq_\mathcal{F}^T \, \bq_\mathcal{F}}}.
$$
The following theorems state some properties of $\bphi(\bx)$, including its
relationship with the projected gradient.
\par \smallskip
\begin{theorem} \label{meas_optim_red_space}
Let $\bx \in \Omega$.~Then $\bphi(\bx) = 0$ if and only if $\bx$ is a stationary point~for
\begin{equation} \label{minomega}
\begin{array}{rl}
\min              & f(\bu), \\[2pt]
\mbox{ s.t.}   & \bu \in \Omega(\bx).
\end{array}	
\end{equation}
\end{theorem}
\vskip -5pt
\begin{proof}
Because of Definition 3.1, $\bphi(\bx)=0$ if and only if
\begin{equation}
\label{phieq0}
\nabla f_i(\bx) - \rho(\bx) \, q_i =0  \quad \forall \, i \in \mathcal{F}(\bx).
\end{equation}
On the other hand, $\bx$ is a stationary point for problem~\eqref{minomega} if and only if
$\nabla f(\bx) = \sum_{i \in \mathcal{A}(\bar\bx)}\nu_i \be_i + \mu \, \bq$,
with $\nu_i, \mu \in \Re$, which implies
\begin{eqnarray} \label{cond1} 
\nabla f_i (\bx) = \mu \, q_i  \quad \forall \, i  \in \mathcal{F}(\bar\bx).
\end{eqnarray}
The thesis follows by comparing~\eqref{phieq0} and~\eqref{cond1}.
\end{proof}

\par\smallskip
\begin{remark} \label{remark_phi}
Theorem~\ref{meas_optim_red_space} shows that $\bphi(\bx)$ can be considered as
a measure of optimality within the reduced space determined by the active variables at $\bx$.
\end{remark}

\par \smallskip
\begin{theorem}\label{lemma:phi_proj_g_par}
For any $\bx \in \Omega$, $\bphi(\bx)$ is the orthogonal projection
of $-\nabla_\Omega f(\bx)$ onto $\Omega_0 (\bx)$, where $\Omega_0 (\bx)$
is given in~\eqref{omega0x}. Furthermore, \\[-15pt]
\begin{equation}\label{free_projected}
    \Vert \bphi(\bx) \Vert^2 = - ( \nabla_\Omega f(\bx))^T\bphi(\bx).
\end{equation}	
\end{theorem}
\begin{proof}
By the definition of projected gradient (see~\eqref{pg_grad_minprob}), \\[-13pt]
\begin{eqnarray}
   & & \qquad \;\;\, (\nabla_\Omega f (\bx) )^T \bq = 0,  \label{pg_prob_kkt2} \\
   & & \nabla_\Omega f (\bx)  = -\nabla f (\bx) + \nu \,\bq + \bmu \label{pg_prob_kkt}
\end{eqnarray}
for some $ \nu \in \Re$ and $\bmu \in \Re^n$, with
$$
\mu_i=0 \mbox{ if } i \in \mathcal{F}(\bx), \quad
\mu_i \geq 0 \mbox{ if } i \in \mathcal{A}_l(\bx), \quad
\mu_i \leq 0 \mbox{ if }  i \in \mathcal{A}_u(\bx).
$$
\par \vskip -8pt
\noindent Let \\[-20pt]
$$
\sigma = \nu - \rho(\bx), \quad 
\tau_i = \mu_i - h_i(\bx) \; \mbox{ if } i \in \mathcal{A}(\bx), \quad
\tau_i = 0 \; \mbox{ if } i \in \mathcal{F}(\bx),
$$
where $\rho(\bx)$ and $\bh(\bx)$ are given in~\eqref{defrho}
and \eqref{defh}, respectively. Then \eqref{pg_prob_kkt} can be written as \\[-18pt]
\begin{eqnarray*}
h_i (\bx) & = & - (\nabla_{\Omega} f)_i(\bx)+ \sigma q_i + \tau_i
        \quad \mbox{if } i \in \mathcal{F}(\bx), \\
0 & = &  - (\nabla_{\Omega} f)_i(\bx) +  \sigma q_i + \tau_i \quad \mbox{if } i \in \mathcal{A}(\bx),
\end{eqnarray*}
\par \vskip -6pt
\noindent or, equivalently, \\[-18pt]
\begin{equation} \label{phi_proj_kkt}
\bphi(\bx) = - \nabla_{\Omega} f(\bx) + \sigma \bq + \btau, 
\end{equation}
with $\tau_i = 0$ if $i \in \mathcal{F}(\bx)$.
%
%
This, with \eqref{pg_prob_kkt2} and $\varphi_i = 0$ for $i \in \mathcal{A}(\bx)$,
proves that
$$
    \bphi(\bx) = 
    \argmin \left\{ \| \bv + \nabla_\Omega f(\bx) \| \; \mbox{ s.t.} \; \bv \in \Omega_0(\bx) \right\},
$$
which is the first part of the thesis. Equation~\eqref{free_projected} follows
from~\eqref{phi_proj_kkt} and the definition of $\bphi(\bx)$.
%
%
\end{proof}

\par\smallskip
\begin{theorem} \label{lemma:phi_binding}
Let $ \bx \in \Omega$. Then $\mathcal{A}(\bx)=\mathcal{B}(\bx)$ if and only if
\begin{equation} \label{AeqB}
     \varphi(\bx)= - \nabla_\Omega f(\bx).
\end{equation}
\end{theorem}
\par \vskip -11pt
\begin{proof}
Assume that $\mathcal{A}(\bx)=\mathcal{B}(\bx)$.
Proving~\eqref{AeqB} means showing that
\begin{equation} \label{}
   - \bphi(\bx) = \argmin \left\lbrace \Vert \bv +
   \nabla f(\bx)\Vert \; \mbox{ s.t.} \; \bv\in T_\Omega(\bx) \right\rbrace.
\end{equation}
Since, by Theorem~\ref{lemma:phi_proj_g_par}, $-\bphi(\bx) \in \Omega_0(\bx)$,
we need only to prove that
$$
    -\bphi (\bx) = -\nabla f (\bx) + \nu \,\bq + \bmu,
$$
for some $ \nu \in \Re$ and $\bmu \in \Re^n$, with $ \mu_i=0 $ if $i \in \mathcal{F}$,
$\mu_i \geq 0 $ if $ i \in \mathcal{A}_l(\bx)$, $\mu_i \leq 0 $ if $ i \in \mathcal{A}_u(\bx)$.
Since $ \mathcal{A}(\bx) = \mathcal{B}(\bx)$, the previous equality holds by setting
$\nu = \rho(\bx)$, $\mu_i = h_i(\bx)$ for $i \in \mathcal{A}(\bx)$, and $\mu_i = 0$ otherwise.

Now we suppose that~\eqref{AeqB} holds. From the definition of $\bphi$ and \eqref{pg_prob_kkt},
it follows that~\eqref{AeqB} can be written as
\begin{eqnarray}
   \varphi_i(\bx)=\nabla f_i(\bx) - \rho(\bx) q_i  =
   \nabla f_i (\bx) - \nu \, q_i \quad \forall \, i \in  \mathcal{F}(\bx),  \label{AeqB1} \\
   0=\nabla f_i (\bx) - \nu \, q_i - \mu_i \quad \forall \, i \in  \mathcal{A}(\bx),\label{AeqB2}
\end{eqnarray}
with $\mu_i \geq 0$ if $i \in \mathcal{A}_l(\bx)$ and  $\mu_i \leq 0$ if $i \in \mathcal{A}_u(\bx)$.
From \eqref{AeqB1} we get $\rho(\bx )= \nu$, and then, from \eqref{AeqB2} and the
definition of $h(\bx)$,
$$
h_i (\bx) \geq 0 \; \mbox{ if } i \in \mathcal{A}_l(\bx), \quad
h_i (\bx) \leq 0 \; \mbox{ if } i \in \mathcal{A}_u(\bx);
$$
thus $\mathcal{A}(\bx) = \mathcal{B}(\bx)$.
\end{proof}

\par \smallskip
Inspired by the two previous lemmas, we give the following definition.
\par \smallskip
\begin{definition}
For any $\bx \in \Omega$, the chopped gradient $\bbeta(\bx)$ is defined as \\[-12pt]
\begin{equation}
\label{definition_beta}
\bbeta(\bx) := - \nabla_\Omega f(\bx) - \bphi(\bx).
\end{equation}
\end{definition}
\par \vskip -12pt
\begin{remark} \label{remark_beta}
Because of Theorem~\ref{lemma:phi_binding}, $\bbeta(\bx) = 0$ if and only if 
$\mathcal{A}(\bx) = \mathcal{B}(\bx)$. Thus, $\bbeta(\bx)$ can be
regarded as a ``measure of bindingness'' of the active variables at~$\bx$.
\end{remark}
\par \smallskip \noindent
Some properties of $\bbeta(\bx)$ are given next.
\par\smallskip
\begin{theorem}
For any $\bx \in \Omega$, $ \bbeta(\bx) $ has the following properties: 
\begin{eqnarray}
\bbeta(\bx) \perp \bphi(\bx), \quad \; \bbeta(\bx) \perp \bq, \label{beta_prop}\\
-\bbeta(\bx) \in T_{\Omega}(\bx). \label{beta_prop1}
\end{eqnarray}
\end{theorem}
\vskip -16pt
\begin{proof}
Since 
$$
\bbeta(\bx)^T \bphi(\bx) = \left(-\nabla_\Omega f (\bx)- \bphi(\bx) \right)^T \! \bphi(\bx)
= (-\nabla_\Omega f(\bx))^T \bphi(\bx) - \bphi(\bx)^T\bphi(\bx),
$$
the first orthogonality condition in  \eqref{beta_prop} follows from \eqref{free_projected}. 
The second one follows from \\[-20pt]
$$ 
\bbeta(\bx)^T \bq = (-\nabla_\Omega f(\bx))^T \bq - \bphi(\bx)^T \bq,
$$
by observing that $\nabla_\Omega f(\bx)$ and $ \bphi(\bx)$ are orthogonal to $\bq$.
Finally, \eqref{beta_prop1} trivially follows from Theorem~\ref{lemma:phi_proj_g_par} and
the definition of $\nabla_{\Omega}f(\bx)$.
\end{proof}

\smallskip
\begin{theorem} \label{lemma:grad_beta_norma_beta}
For any $\bx \in \Omega$, $\Vert \bbeta(\bx) \Vert^2 = \nabla f(\bx)^T\bbeta(\bx)$.
\end{theorem}
\begin{proof}
By \cite[Lemma~3.1]{CalamaiMore:1987}, we have
$ - (\nabla f (\bx))^T \nabla_\Omega f(\bx)
= \Vert \nabla_\Omega f (\bx) \Vert^2$, which can be written as
\begin{equation} \label{phi_beta}
(\nabla f(\bx) )^T \! \left(\bphi(\bx)  + \bbeta(\bx) \right)
= \Vert \bphi(\bx)  \Vert^2 + \Vert \bbeta(\bx)  \Vert^2
\end{equation}
by exploiting~\eqref{definition_beta} and \eqref{beta_prop}.
We note that the scalar product $(\nabla f(\bx) )^T \bphi(\bx) $ involves only the entries
corresponding to $\mathcal{F}(\bx)$. Furthermore, since $ \bphi_\mathcal{F} (\bx) = 
\nabla f_\mathcal{F}(\bx)  - \rho (\bx) \bq_\mathcal{F} $,
where $\rho(\bx)$ is given in~\eqref{defrho}, we get
\begin{eqnarray*}
(\nabla f(\bx) )^T \bphi (\bx)  & = &
\Vert \nabla f_{\mathcal{F}} (\bx) \Vert^2 - \rho(\bx) \bq_{\mathcal{F}}^T\nabla f_{\mathcal{F}}(\bx), \\
\Vert \bphi (\bx) \Vert^2 & = & \Vert \nabla f_{\mathcal{F}}(\bx) \Vert^2 - 2\,\rho (\bx) \bq_{\mathcal{F}}^T\nabla f_{\mathcal{F}}(\bx) + \rho^2 \Vert \bq_{\mathcal{F}} (\bx) \Vert^2.
\end{eqnarray*}		
By subtracting the two equations and using the expression of $\rho (\bx)$, we get
$$
(\nabla f (\bx))^T (\bx) - \Vert \bphi(\bx) \Vert^2 = 0;
$$
then the thesis follows from~\eqref{phi_beta}.
\end{proof}


\subsection{Proportional iterates for SLBQPs\label{sec:prop_iters}}

So far we managed to decompose the projected gradient $\nabla_\Omega f(\bx)$
into two parts: $-\bphi(\bx)$, which provides a measure of stationarity within the
reduced space determined by the active variables at $\bx$,
%
%
and $-\bbeta(\bx)$, which gives a measures of bindingness of the active variables
at $\bx$. With this decomposition we can apply to problem~\eqref{SLBQP}
the definition~\eqref{proportioning} of proportional iterates introduced
for the BQP case. In the strictly convex case, disproportionality of $\bx^k$ again
guarantees that the solution of~\eqref{SLBQP} does not belong to the
face identified by the active variables at $\bx^k$. This result is a consequence
of the next theorem, which generalizes Theorem~3.2 in \cite{Dostal:1997}
and is the main result of this section.
\par \smallskip
\begin{theorem}\label{th:Dostal3.2}
Let $H$ be the Hessian matrix in~\eqref{SLBQP} and let $H_q = V^T H \, V$ be
positive definite, where $V  \in \Re^{n \times  (n-1)}$ has orthonormal columns spanning
$\left\lbrace \bq \right\rbrace^\perp$.
Let $\bx\in \Omega$ be such that $\Vert \bbeta(\bx) \Vert_\infty > \kappa(H_q)^{1/2}\, 
\Vert \bphi(\bx) \Vert_2 $, and let $\bar{\bx}$ be the solution of \\[-8pt]
\begin{equation}
\label{subp_theorem}
\begin{array}{rl}
\min             & f(\bu), \\[2pt]
\mbox{\rm s.t.}   & \bu \in \Omega(\bx),
\end{array}
\end{equation}
\par \vskip -1pt
\noindent		
where $\Omega(\bx)$ is defined in~\eqref{omegax}. If $ \bar{\bx}\in\Omega$, then $\bbeta(\bar{\bx})\neq 0$.
\end{theorem}
\par \smallskip \noindent
To prove Theorem~\ref{th:Dostal3.2}, we need the lemma given next.
\par \smallskip
\begin{lemma}\label{lemma:PqHF_pseudoinverse}
Let us consider the minimization problem
\begin{equation}\label{PROB:SmallLinConQP}
\begin{array}{rl}
\min& w(\bz) := \frac{1}{2} \bz^T\,A\,\bz - \bp^T\bz , \\
\mbox{\rm s.t.}   & \br^T\bz = s,
\end{array}	
\end{equation}
where $A \in \Re^{m \times m}$, $\bp, \br \in \Re^{m}$, $s \in\Re$, and $m \ge 1$.
Let $\Theta = \left\{ \bz\in\Re^m \, : \, \br^T\bz = s \right\}$ and
$\Theta_0 = \left\{ \bz\in\Re^m \, : \, \br^T\bz = 0 \right\}$. Let $P_{\Theta_0}$
be the orthogonal projection onto $\Theta_0$, and $U \in \Re^{m \times  (m-1)}$ a
matrix with orthonormal columns spanning $\Theta_0$. Finally, let $U^T A \, U$ be positive
definite, and $\bar{\bz}$ the solution of \eqref{PROB:SmallLinConQP}. 
Then
\begin{equation}\label{EQN:Pq_H_invertible}
\bz-\bar{\bz} = B\,P_{\Theta_0} \nabla w(\bz), \quad \forall \bz\in\Theta,
\end{equation}
where $B = U (U^T A U)^{-1} U^T$. Furthermore,
\begin{equation}\label{EQN:funcdiff_B}
   w(\bz) - w(\bar \bz)
   \leq \frac{1}{2} \, \Vert B \Vert  \Vert P_{\Theta_0} \nabla w(\bz) \Vert^2.
\end{equation}
\end{lemma}
\par \vskip -9pt
\begin{proof}
Without loss of generality we assume $\Vert \br \Vert_2 = 1$. Let $\bz\in\Theta$;
since $s = \br^T\bz$ and $\range(U)$ is the space orthogonal to $\br$, we have
$$
\bz = s\, \br + U \by,
$$
for some $\by \in \Re^{m-1}$.
%
%
Thus, \eqref{PROB:SmallLinConQP} can be reduced to
\begin{equation*}
\label{EQN:ReducedPB}
\min \; \widetilde w(\by) := \frac{1}{2} \by^T U^T A \, U \by - (\bp^T - s\,\br^T A) U \by.
\end{equation*}
By writing $\bar \bz$, the minimizer of \eqref{PROB:SmallLinConQP},
as $\bar \bz = s \, \br + U \bar \by$, we have
\begin{equation}
\label{EQN:DiffZ}
\bz - \bar \bz = U (\by - \bar \by)
\end{equation}
and, by observing that $\nabla \widetilde w(\bar \by) = 0$,  we obtain
\begin{equation}
\label{EQN:DiffNablaTildew}
\nabla \widetilde w(\by) = \nabla \widetilde w(\by) - \nabla \widetilde w(\bar \by) =
U^T A \, U ( \by - \bar \by ) = U^T ( \nabla w(\bz) - \nabla w(\bar \bz) ).
\end{equation}
%
%
Since $\nabla w(\bar \bz) = \gamma \, \br$ for some $\gamma \in \Re$,
we get $U U^T \nabla w(\bar \bz)  = P_{\Theta_0} \nabla w(\bar \bz) = 0$ and hence
\begin{equation}
\label{EQN:VDiffNablaTildew}
U \nabla \widetilde w(\by) = U U^T ( \nabla w(\bz) - \nabla w(\bar \bz) ) =
P_{\Theta_0} \nabla w(\bz).
\end{equation}
From \eqref{EQN:DiffZ}, \eqref{EQN:DiffNablaTildew} and \eqref{EQN:VDiffNablaTildew}
it follows that
$$
\bz - \bar \bz = U (\by - \bar \by) = U (U^T A \, U)^{-1} U^T U \nabla \widetilde w(\by) =
B \, P_{\Theta_0} \nabla w(\bz),
$$
which is \eqref{EQN:Pq_H_invertible}.

Let $\phi(\bz) = P_{\Theta_0} \nabla w(\bz)$. By applying \eqref{EQN:Pq_H_invertible},
we get
$$
w(\bz) - w(\bar \bz) =\frac{1}{2} (\bz - \bar \bz)^T A   (\bz - \bar \bz) =
\frac{1}{2} \phi(\bz)^T B^T \, A \, B \, \phi(\bz).
$$
By observing that $B^T A \, B  = U (U^T A U)^{-1} U^T A U (U^T A U)^{-1} U^T  
= B$, we have
$$
w(\bz) - w(\bar \bz)  =  \frac{1}{2} \phi(\bz)^T B \,\phi(\bz) \le \frac{1}{2} \Vert B\Vert\Vert \phi(\bz)\Vert^2,
$$
which completes the proof.
\end{proof}

\par\smallskip
Now we are ready to prove Theorem~\ref{th:Dostal3.2}.
\medskip
\begin{proofof}{Theorem~\ref{th:Dostal3.2}}
Let $\by = \bx - \Vert H_q \Vert^{-1}\,\bbeta(\bx)$. By Theorem~\ref{lemma:grad_beta_norma_beta}
and observing that $\Vert \cdot \Vert \geq \Vert \cdot \Vert_\infty$ and $\bbeta(\bx) = V V^T\bbeta(\bx)$,
because $\bbeta(\bx)\in\left\lbrace \bq \right\rbrace^\perp$, we get
\begin{eqnarray}
   f(\by)-f(\bx) & = & \frac{1}{2}\,\Vert H_q \Vert^{-2}\,\bbeta(\bx)^T H \bbeta(\bx) - 
          \Vert H_q \Vert^{-1}\, (\nabla f(\bx))^T\bbeta(\bx) \nonumber\\
   & = & \frac{1}{2}\,\Vert H_q \Vert^{-2}\,\bbeta(\bx)^T V\,H_q\,V^T \bbeta(\bx) - 
          \Vert H_q \Vert^{-1}\,\Vert \bbeta(\bx) \Vert^2 \nonumber\\
   & \leq & \frac{1}{2}\,\Vert H_q \Vert^{-1}\,\Vert V^T\bbeta(\bx) \Vert^2 - 
	\Vert H_q \Vert^{-1}\,\Vert \bbeta(\bx) \Vert^2 = 
	  -\frac{1}{2}\,\Vert H_q \Vert^{-1}\,\Vert \bbeta(\bx) \Vert^2 \nonumber\\
   &< &  -\frac{1}{2}\,\Vert H_q \Vert^{-1}\,\kappa(H_q)\,\Vert \bphi(\bx) \Vert^2 
          = -\frac{1}{2}\,\Vert H_q^{-1} \Vert\,\Vert \bphi(\bx) \Vert^2. \label{theorem_DOST_y-x}
\end{eqnarray}
%
%
The point $\bar\bx$ satisfies the KKT conditions of problem~\eqref{subp_theorem},
\begin{eqnarray}
& & \;\, \nabla f(\bar{\bx}) = \sum_{i\in \mathcal{A}(\bx)}\eta_i\be_i+ \gamma\,\bq, 
\label{kkt_subp} \\
& & \bq^T\bar{\bx} = b,  \;\;  \bar{x}_i = x_i \;\,  \forall \, i \in \mathcal{A}(\bx), \nonumber
\end{eqnarray}
where $\eta_i$ and $\gamma$ are the Lagrange multipliers, and hence
\begin{eqnarray}
   \nabla f(\bar{\bx})^T (\bx-\bar{\bx}) & = &
       \sum_{i\in \mathcal{A}} \left( \eta_i \be_i + \gamma \bq \right)^T \! (\bx-\bar{\bx})=0,
       \label{gradbarort} \\
       \nabla f_{\mathcal{F}} ( \bar{\bx} ) & = &\gamma \, \bq_\mathcal{F},
    \label{gradpropq}
\end{eqnarray}
where $\mathcal{A} = \mathcal{A} (\bx)$ and $\mathcal{F} = \mathcal{F} (\bx)$.
It follows that
\begin{equation} \label{theorem_DOST_x-barx}
   f(\bx)-f(\bar{\bx}) = \frac{1}{2}\,(\bx-\bar{\bx})^T H (\bx-\bar{\bx}) + \nabla f(\bar{\bx})^T(\bx-\bar{\bx})		                        
                              = \frac{1}{2}\,(\bx-\bar{\bx})_\mathcal{F}^T\,H_{\mathcal{FF}}\, (\bx-\bar{\bx})_\mathcal{F}.
\end{equation}
Now we apply Lemma~\ref{lemma:PqHF_pseudoinverse}
with $\bz = \bx_\mathcal{F}$, $A = H_\mathcal{FF}$, 
$\bp = \bc_\mathcal{F} - H_{\mathcal{FA}}\,\bx_\mathcal{A}$,
$\br = \bq_\mathcal{F}$, $s= b-\bq_\mathcal{A}^T\,\bx_\mathcal{A}$,
$\Theta_0 = \left\lbrace \bq_{\mathcal{F}}\right\rbrace^\perp$, and $w(\bz)$ defined
as in~\eqref{PROB:SmallLinConQP}.
By~\eqref{phi_F as projection}, we have
$$
P_{\Theta_0} \nabla w(\bz) = P_{\{\bq_{\mathcal{F}} \}^{\bot}}
\left(\nabla f_{\mathcal{F}}(\bx)\right) = \bphi_\mathcal{F} (\bx).
$$
Therefore, from \eqref{EQN:funcdiff_B} and \eqref{theorem_DOST_x-barx} we get
\begin{equation}\label{theorem_DOST_x-barx_2nd_withB}
   f(\bx)-f(\bar{\bx}) 
   \leq \frac{1}{2}\left\Vert B \right\Vert\,\Vert \bphi_\mathcal{F} (\bx) \Vert^2,
\end{equation}
where $B = W (W^T H_{\mathcal{FF}} W)^{-1} W^T$ and $W\in \Re^{|\mathcal{F}| \times  (|\mathcal{F}|-1)}$ 
has orthonormal columns spanning $\left\lbrace \bq_\mathcal{F} \right\rbrace^\perp$.
We note that
\begin{equation}\label{theorem_DOST_normB_mineig}
\Vert B\Vert
\le \Vert (W^T H_{\mathcal{FF}} W)^{-1}\Vert =
\zeta_{max} \left((W^T H_{\mathcal{FF}} W)^{-1}\right) = \frac{1}{\zeta_{min} (W^T H_{\mathcal{FF}} W)};
\end{equation}
\noindent furthermore,
\begin{eqnarray}
\!\!\!\!\!\!\!\!\! \zeta_{min} (W^T H_{\mathcal{FF}} W) & \; = &
           \min_{\scriptsize \begin{array}{c} \bs\in\Re^{|\mathcal{F}|-1} \\ \bs \ne 0 \end{array}}
           \frac{\bs^T W^T H_{\mathcal{FF}} W \bs}{\bs^T \bs} \; =
           \!\!\! \min_{\scriptsize \begin{array}{c} \bw \in \Re^{|\mathcal{F}|}, \, \bw \ne 0\\  \bw \perp 
                        \bq_{\mathcal{F}} \end{array}}   \!\!\! \frac{\bw^T H_{\mathcal{FF}} \bw}{\bw^T \bw}                                          
                                                                                                                                           \nonumber \\
   & \; = & \!\! \min_{\scriptsize \begin{array}{c} \bv \in \Re^n,\, \bv \ne 0 \\
           \bv_{\mathcal{F}}\perp\bq_{\mathcal{F}},\;\bv_\mathcal{A} = 0 \end{array}}
           \!\!\!\!\! \frac{\bv^T H \bv}{\bv^T \bv} \; \ge 
           \!\! \min_{\scriptsize \begin{array}{c} \bv \in \Re^n,\, \bv \ne 0\\  \bv\perp\bq \end{array}} 
           \!\!\! \frac{\bv^T H \bv}{\bv^T \bv}                                                                       \label{zetamin} \\
   & \; = & \! \min_{\scriptsize \begin{array}{c} \bu\in\Re^{n-1} \\ \bu \ne 0\end{array}} 
         \!\!\!\! \frac{\bu^T V^T H V\bu}{\bu^T V^T V \bu} \; = \; \zeta_{min}(H_q).                 \nonumber
\end{eqnarray}

\noindent The last inequality, together with \eqref{theorem_DOST_x-barx_2nd_withB} and \eqref{theorem_DOST_normB_mineig}, yields
\begin{equation}\label{theorem_DOST_x-barx_2nd}
f(\bx)-f(\bar{\bx}) 
\leq \frac{1}{2}\frac{1}{\zeta_{min}(H_q)}\Vert \bphi_\mathcal{F} (\bx) \Vert^2 = \frac{1}{2}\Vert H_q^{-1} \Vert\,\Vert \bphi (\bx) \Vert^2. 
\end{equation}
Then, by \eqref{theorem_DOST_y-x} and \eqref{theorem_DOST_x-barx_2nd}, we get
\begin{equation} \label{fyminfbar}
f(\by) - f(\bar\bx) = f(\by) -  f(\bx) + f(\bx) - f(\bar{\bx}) < 0.
\end{equation}

\noindent For the remainder of the proof we assume that 
$ \bar{\bx}\in\Omega $ and set $\bar{\mathcal{F}} := \mathcal{F}(\bar{\bx})$.  
From \eqref{gradpropq} and $\bar{\mathcal{F}} \subseteq \mathcal{F}$ it follows that 
$\nabla f_{\bar{\mathcal{F}}} (\bar{\bx}) = \gamma\,\bq_{\bar{\mathcal{F}}}$, and hence
$$
   \bphi_{\bar{\mathcal{F}}} (\bar{\bx}) =
   \bh_{\bar{\mathcal{F}}} (\bar{\bx}) =
   \nabla f_{\bar{\mathcal{F}}} (\bar{\bx}) -
   \frac{\nabla f_{\bar{\mathcal{F}}} (\bar{\bx})^T\, \bq_{\bar{\mathcal{F}}}}
          {\bq_{\bar{\mathcal{F}}}^T \bq_{\bar{\mathcal{F}}}}\,\bq^{\bar{\mathcal{F}}} = 0.
$$
\par \vskip -5pt
\noindent Therefore \\[-19pt]
\begin{equation}
\bphi(\bar{\bx})=0.
\label{fibarx0}
\end{equation}
By using \eqref{fyminfbar} we get \\[-16pt]
$$ 
   0>f(\by)-f(\bar{\bx}) = \nabla f(\bar{\bx})^T\,(\by-\bar{\bx}) + 
  \frac{1}{2}(\by-\bar{\bx})^T\,H\,(\by-\bar{\bx}) > \nabla f(\bar{\bx})^T\,(\by-\bar{\bx}).
$$
Because of the definition of $\by$ and \eqref{gradbarort}, we have
\begin{eqnarray*} 
   \nabla f(\bar{\bx})^T\,(\by-\bar{\bx}) & = & \nabla f(\bar{\bx})^T\,(\by-\bx) + 
          \nabla f(\bar{\bx})^T\,(\bx-\bar{\bx}) = \nabla f(\bar{\bx})^T\,(\by-\bx)  \\	
   & = & - \Vert H_q^{-1} \Vert\, \nabla f(\bar{\bx})^T\,\bbeta(\bx),
\end{eqnarray*}
and thus \\[-18pt]
\begin{equation} \label{gradf_beta}
\nabla f(\bar{\bx})^T\,\bbeta(\bx)>0 .
\end{equation}

By contradiction, suppose that $\bbeta(\bar{\bx}) = 0$.
Since $\bar\bx \in \Omega$, from~\eqref{fibarx0} it follows that $\bar{\bx}$ is the
optimal solution of problem~\eqref{SLBQP}, and thus $-\nabla f(\bar{\bx})
\in T_\Omega(\bar{\bx})^\circ$.  We consider two cases.
\smallskip
\begin{itemize}[leftmargin=.65cm]
\item[(a)]
$\mathcal{A}(\bx) = \mathcal{A}(\bar{\bx})$. In this case $T_\Omega(\bar{\bx})^\circ
= T_\Omega(\bx)^\circ$, and, since $-\bbeta(\bx)\in T_\Omega(\bx)$ and
$-\nabla f(\bar{\bx}) \in T_\Omega(\bx)^\circ$, it is
$-\nabla f(\bar{\bx})^T\,(-\bbeta(\bx))\leq 0$. This contradicts~\eqref{gradf_beta}.
\vskip 6pt
\item[(b)]
$\mathcal{A}(\bx) \subsetneq \mathcal{A}(\bar{\bx})$. In this case the optimality of
$\bar{\bx}$ for problem~\eqref{SLBQP} yields \\[-14pt]
\begin{equation}
\nabla f(\bar{\bx})
	= \sum_{i \in \mathcal{A}(\bar{\bx})} \lambda_i\be_i + \theta\,\bq,
	\quad \lambda_i \geq 0 \;\; \mbox{if} \;\;  i \in \mathcal{A}_l(\bar{\bx}),
	\quad \lambda_i \leq 0 \;\; \mbox{if} \;\; 
	 i \in \mathcal{A}_u(\bar{\bx}).
	 \label{kkt_prob}
\end{equation}		
Since $\mathcal{F}(\bar{\bx})\subsetneq \mathcal{F}(\bx)$, by comparing
\eqref{kkt_subp} and \eqref{kkt_prob} we find that
$\nabla f_i (\bar{\bx})$ $= \theta q_i = \gamma q_i$ for all $i \in \mathcal{F}(\bar{\bx})$,
and thence $\theta=\gamma$. Then, $\eta_i=\lambda_i$ for
$i \in \mathcal{A}(\bx)$, whereas $\lambda_i = 0$ for 
$i \in \mathcal{A}(\bar{\bx}) \setminus \mathcal{A}(\bx)$, i.e., \\[-12pt]
\begin{equation*}
\nabla f(\bar{\bx})
= \sum_{i \in \mathcal{A}(\bx)} \lambda_i\be_i + \theta\,\bq,
\quad \lambda_i \geq 0 \; \mbox{ if } \;  i \in \mathcal{A}_l(\bx),
\quad \lambda_i \leq 0 \; \mbox{ if } \; i \in \mathcal{A}_u(\bx).
\end{equation*}
\par \vskip -1pt
\noindent Therefore $-\nabla f(\bar{\bx}) \in T_\Omega(\bx)^\circ$, which leads
to a contradiction as in case (a).
\end{itemize}
\vspace*{-3.8mm} \hfill \end{proofof}

\section{Proportionality-based 2-phase Gradient Projection method\label{sec:method}}

Before presenting our method, we briefly describe the basic GP method as stated
by Calamai and Mor\'{e} in \cite{CalamaiMore:1987}. Given the current iterate
$\bx^k$, the next one is obtained as 
$$
  \bx^{k+1}=P_{\Omega}(\bx^k - \alpha^k \nabla f^k),
$$
where $P_\Omega$ is the orthogonal projection onto $\Omega$, and
$\alpha^k$ satisfies the following sufficient decrease condition:
given $\gamma_1,\gamma_2,\gamma_3>0$ and $\mu_1,\mu_2\in(0,1)$,
\begin{equation}\label{SuffDecrCond1}
	f^{k+1} \leq f^k+\mu_1\,(\nabla f^k)^T(\bx^{k+1}-\bx^k),
\end{equation}
where \\[-18pt]
\begin{equation}\label{SuffDecrCond2}
\begin{array}{c}
\displaystyle \qquad \alpha^k \leq \gamma_1,\\[3pt]
\displaystyle \alpha^k \ge \gamma_2\;\mbox{ or }\; \alpha^k\geq\gamma_3\,\bar{\alpha}^k>0,
\end{array}
\end{equation}
\par \vskip -2pt
\noindent with $\bar{\alpha}^k$ such that
\begin{equation}\label{SuffDecrCond3}
    f(\bx^k(\bar{\alpha}^k)) > f^k+\mu_2\,(\nabla f^k)^T(\bx^k(\bar{\alpha}^k)-\bx^k),
\end{equation}
where $\bx^k(\bar{\alpha}^k) : = P_{\Omega}(\bx^k - \bar{\alpha}^k \nabla f(\bx^k))$. 
In Section~\ref{sec:identification} a simple practical procedure is described
for the determination of $\alpha_k$ that satisfies the sufficient decrease condition.

In \cite[Algorithm~5.3]{CalamaiMore:1987} a very general algorithmic framework is
presented, where the previous GP steps are used in selected iterations, alternated
with simple decrease steps aimed to speedup the convergence of the overall algorithm.
%
%
The role of GP steps is to identify promising active sets, i.e., active variables that
are likely to be active at the solution too. Once a suitable active set has been fixed
at a certain iterate~$\bx^k$, a reduced problem is defined on the complementary
set of free variables
\begin{equation}\label{reducedsubp}
\begin{array}{rl}
\min                    & f(\bx), \\[2pt]
\mbox{\rm s.t.}   & \bx \in \Omega(\bx^k),
\end{array}
\end{equation}
Problem~\eqref{reducedsubp} can be easily formulated as an unconstrained
quadratic problem, as shown in Section~\ref{sec:identification}.

We now introduce the Proportionality-based 2-phase Gradient Projection (P2GP) method
for problem~\eqref{SLBQP}. The method does not assume that~\eqref{SLBQP} is
strictly convex. However, if~\eqref{SLBQP} is not strictly convex,
the method only computes an approximation of a stationary point or finds that
the problem is unbounded below. If strict convexity holds, P2GP provides an approximation
to the optimal solution. 
The method is outlined in Algorithm~\ref{alg:P2GP} and explained in detail in the next sections.
For the sake of brevity, $\bphi(\bx^k)$ and $\bbeta(\bx^k)$ are denoted by
$\bphi^k$ and $\bbeta^k$, respectively.
Like GPCG, it alternates identification phases, where GP steps are performed that satisfy
\eqref{SuffDecrCond1}-\eqref{SuffDecrCond3},
and minimization phases, where an approximate solution to~\eqref{reducedsubp} is
searched, with $\bx^k$ inherited from the last identification phase.
Unless a point satisfying
\begin{equation} \label{accur_req}
\| \bphi^k + \bbeta^k \| \le tol
\end{equation}
is found, or the problem is discovered to be unbounded below, the identification phase
proceeds either until a promising active set
$\mathcal{A}^{k+1}$ is identified (i.e., an active set that remains fixed in two consecutive
iterations) or no reasonable progress is made in reducing the objective function, i.e., \\[-14pt]
\begin{equation} \label{stopGP}
f^k - f^{k+1} \le \eta\, \max_{m \le l < k} (f^l - f^{l+1}),
\end{equation}
\par \vskip -3pt
\noindent where $\eta$ is a suitable constant and $m$ is the first iteration of the
current identification phase. This choice follows that in~\cite{More:1991}.
In the minimization phase, an approximate solution to the reduced
problem obtained by fixing the variables with indices in the current active set is searched for.
The proportionality criterion~\eqref{proportioning} is used to decide when the minimization
phase has to be terminated; this is a significant difference from the GPCG method, which
exploits a condition based on the bindingness of the active variables.
Note that the accuracy required in the solution of the reduced problem~\eqref{reducedsubp}
affects the efficiency of the method and a loose stopping criterion must be used, since the
control of the minimization phase is actually left to the proportionality criterion
(more details are given in Section~\ref{sec:minimization}). Like the identification,
the minimization phase is abandoned if a suitable approximation to a stationary point
is computed or unboundedness is discovered. Nonpositive curvature directions
are exploited as explained in Sections~\ref{sec:identification}
and~\ref{sec:minimization}.
	
\begin{algorithm}[h!]
\small
\caption{(P2GP) \label{alg:P2GP}}
\begin{algorithmic}[1]
\vskip 2pt
\State {$x_0 \in \Omega; \;\; tol \geq 0; \;\; \eta \in (0, 1); \;\; \Gamma > 0; \;\; k = 0;$}
\State {$conv = (\left \Vert \bphi^k+\bbeta^k \right \Vert \le {tol}); \;\; unbnd = .f\!alse.; \;\; 
           phase1 = .true.; \;\; phase2 = .true.$}
\While {$(\neg \, conv$  $\wedge$ $\neg \, unbnd)$}  \Comment{\textsc{Main loop}}
                                                                                                      \label{alg:main_loop_start}
	\State {$m = k$;}                                                                            \label{alg:ident_start}
	\While {($phase1$)}   \Comment{\textsc{Identification Phase}}
        \If {($(\nabla_\Omega f^k)^T H \, (\nabla_\Omega f^k)\leq 0 \; \wedge \;
                 \max \left\lbrace \alpha > 0 \; \mbox{ s.t.} \;\,
                 \bx^k+\alpha \nabla_\Omega f^k\in\Omega \right\rbrace = +\infty$)}
        \State {$unbnd = .true.;$}
	\Else
	\State {$\bx^{k+1}=P_{\Omega}(\bx^k-\alpha^k \nabla f^k)$
	      with $\alpha^k$ such that \eqref{SuffDecrCond1}-\eqref{SuffDecrCond3} hold;} 
                                                                                                                \label{alg:proj1}
	\EndIf
        \If {($\neg \, unbnd$)}
        \State {$conv = (\left \Vert \bphi^{k+1} +\bbeta^{k+1} \right \Vert \le {tol});$}
        \vspace*{-.5mm}
        \State {$phase1 = (\mathcal{A}^{k+1}  \ne \mathcal{A}^k$)  $\wedge$
		($f^k - f^{k+1}  > \eta\, \displaystyle \max_{m \le l < k} (f^l - f^{l+1})$) $\wedge$
		$(\neg \, conv);$}
        \vspace*{-1.5mm}
        \State  $k=k+1$;
        \EndIf
        \EndWhile                                                                                     \label{alg:ident_end}
        \If {($conv$  $\vee$ $unbnd$)}
       \State {$phase2 = .f\!alse.;$}
       \EndIf
	\While {($phase2$)}	 \Comment{\textsc{Minimization Phase}}  \label{alg:minim_start}
	\State {Compute an approx.~solution~$\bd^k$ to}
                  $\min \! \left\{ \! f(\bx^k \! + \bd) \;\, \mbox{s.t.} \; \bq^T\bd=0, \,
                    d_i=0 \mbox{ if } i \in \mathcal{A}^k \! \right\} \! ;$
	\If {($(\bd^k)^T H \, \bd^k \leq 0$)}
	\State {Compute $\alpha^k = \max \left\lbrace \alpha > 0 \; \mbox{ s.t.} \;\,
                \bx^k+\alpha \bd^k\in\Omega \right\rbrace;$}
	\If {($\alpha = +\infty$)}
        \State {$unbnd = .true.;$}
	\Else
	\State {$\bx^{k+1} = \bx^k+\alpha^k \bd^k$;}
        \State {$conv = (\left \Vert \bphi^{k+1} +\bbeta^{k+1} \right \Vert \le {tol});$}
                                                                                                              \label{alg:gp_conv}
	\State {$k=k+1$;}
	\EndIf
        \State {$phase2 = .f\!alse.$;}
	\Else
	\State {$\bx^{k+1} = P_{\Omega^k}({\bx}^k+\alpha^k \bd^k)$
		with $\alpha^k$ such that $f^{k+1} < f^k$ and  
               $\Omega^k = \Omega\,\cap\,\Omega(\bx^k)$} \label{alg:proj2}
        \State $conv = (\left \Vert \bphi^{k+1}+\bbeta^{k+1} \right \Vert \le {tol});$
        \State {$phase2 = (\Vert \bbeta^{k+1} \Vert_\infty \le \Gamma \,
		\Vert \bphi^{k+1} \Vert_2$)  $\wedge$ ($\neg \, conv$);}
        \State {$k = k+1$;}
	\EndIf
        \EndWhile                                                                          \label{alg:minim_end}
        \State {$phase1 = .true.;$ \ $phase2 = .true.;$}
\EndWhile
\If {($conv$)}
       \State {\Return $\bx^k$}
\Else
       \State {\Return ``problem~\eqref{SLBQP} is unbounded'';}
\EndIf
\end{algorithmic}
\end{algorithm}

We note that the minimization phase can add variables to the active set,
but cannot remove them, and thus P2GP fits into the general framework
of~\cite[Algorithm~5.3]{CalamaiMore:1987}. Thus we may exploit
general convergence results available for that algorithm. To this end,
we introduce the following definition.
\par \smallskip
\begin{definition}
Let $\left\{\bx^k\right\}$ be a sequence generated by the P2GP method
applied to problem~\ref{SLBQP}. The set \\[-15pt]
$$
  K_{GP} = \left\lbrace k\in\mathbb{N} \; :\; \bx^{k+1} 
  \mbox{ is generated by step~\ref{alg:proj1} of Algorithm~\ref{alg:P2GP}} \right\rbrace
$$
\par \vskip -3pt
\noindent is called set of GP iterations.	
\end{definition}
 \par \smallskip \noindent 
The following convergence result holds, which follows
from~\cite[Theorem~5.2]{CalamaiMore:1987}.
\par \smallskip
\begin{theorem}\label{th:Conv_result_1}
Let $\left\{\bx^k\right\}$ be a sequence generated by applying the P2GP method
to problem~\eqref{SLBQP}. Assume that the set of GP iterations, $K_{GP}$, is infinite.
If some subsequence  $\left\{\bx^k\right\}_{k\in K}$, with $K\subseteq K_{GP}$,
is bounded, then
\begin{equation}\label{ProjGradNorm_tends_to_0}
\lim\limits_{k\in K,\,k\rightarrow\infty} \left\Vert \nabla_\Omega f(\bx^{k+1})\right\Vert = 0.
\end{equation}
Moreover, any limit point of $\left\{\bx^k\right\}_{k\in K_{GP}}$ is a stationary point
for problem~\eqref{SLBQP}. 
%
%
\end{theorem}
\par \smallskip

The identification property of the GP steps is inherited by the whole sequence generated by the 
P2GP method, as shown by the following Lemma.
\par \smallskip
\begin{lemma} \label{lemma:P2GP_identifies_A1}
Let us assume that problem~\eqref{SLBQP} is strictly convex
and $\bx^*$ is its optimal solution. If $\left\{\bx^k\right\}$ is a sequence
in $\Omega$ generated by the P2GP method applied to~\eqref{SLBQP}, then
for all $k$ sufficiently large
$$
    \mathcal{A}_N^* \subseteq \mathcal{A}^k \subseteq \mathcal{A}^*
$$
where $\mathcal{A}_N^*$ is defined in Theorem~\ref{th:convres_calamaimore_gptozero}.
\end{lemma}
\begin{proof}
Since $f(\bx)$ is bounded from below and the sequence  $\left\lbrace f^k \right\rbrace$
is decreasing, the sequence $\left\{\bx^k\right\}$ is bounded,
and, because of Theorem \ref{th:Conv_result_1}, there is a subsequence
$\left\lbrace \bx^k\right\rbrace_{k \in K^*}$, with $K^*\subseteq K_{GP}$, which
converges to $\bx^*$. Now we show that the whole sequence $\left\lbrace \bx^k \right\rbrace$ 
converges to $\bx^*$. For any $k\in \mathbb{N}$ we have \\[-3.4mm]
\begin{equation}\label{k_kplus}
f(\bx^k)-f(\bx^*) \leq f(\bx^{k^+})-f(\bx^*),
\end{equation}
\noindent where $k^+ = \min\left\lbrace s\in K^* \; : \; s\geq k \right\rbrace$. Moreover, for the
stationarity of $x^*$ we have
$ \nabla f(\bx^*)^T\,(\bx^k - \bx^*)\geq 0$,
and then \\[-15pt]
\begin{equation}\label{eq:fx-fx*}
\begin{array}{ll}
   f(\bx^k)-f(\bx^*) \!\!\!\!\!  & \displaystyle \; = \nabla f(\bx^*)^T\,(\bx^k - \bx^*) +  
                                      \frac{1}{2}(\bx^k - \bx^*)^T\,H\,(\bx^k - \bx^*) \\
                 & \; \ge \displaystyle \frac{1}{2} (\bx^k - \bx^*)^T V H_q V^T (\bx^k - \bx^*)         
                     \ge \zeta_{min}(H_q)\, \Vert \bx^k - \bx^*\Vert^2,
\end{array}
\end{equation}
where $H_q$ and $V$ are defined in Theorem~\ref{th:Dostal3.2} and
the equality $\bx^k - \bx^* = VV^T (\bx^k - \bx^*)$ has been exploited.
From~\eqref{k_kplus} and \eqref{eq:fx-fx*} it follows that $\left\lbrace \bx^k\right\rbrace$
converges to $\bx^*$. Then, for $k$ sufficiently large, $\mathcal{F}^* \subseteq \mathcal{F}^k$
and hence $\mathcal{A}^k \subseteq \mathcal{A}^*$.
Furthermore, by Theorem~\ref{th:convres_calamaimore_gptozero}, the convergence of
$\left\lbrace \bx^k \right\rbrace_{k\in K_{GP}}$ to $\bx^*$,
together with \eqref{ProjGradNorm_tends_to_0}, yields $\mathcal{A}_N^* \subseteq
\mathcal{A}(\bx^k)$ for all $k\in K_{GP}$ sufficiently large. Since minimization
steps do not remove variables from the active set, we have  $\mathcal{A}_N^* \subseteq
\mathcal{A}(\bx^k)$ for all $k$ sufficiently large.
\end{proof}

\par \smallskip
We note that in case of nondegeneracy ($\mathcal{A}_N^* = \mathcal{A}^*$)
the active set eventually settles down, i.e., the identification property holds.
This implies that the the solution of~\eqref{SLBQP} reduces to the solution of an
unconstrained problem in a finite number of iterations, which is the key ingredient
to prove finite convergence of methods
that fit into the framework of \cite[Algorithm~5.3]{CalamaiMore:1987}, such as
the GPCG one. In case of degeneracy we can just say that the nondegenerate
active constraints at the solution will be identified in a finite number of steps. 
However, in the strictly convex case, finite convergence can be achieved in this case
too, provided a suitable value of $\Gamma$ is taken, as stated by the following theorem.
\par \smallskip
\begin{theorem}\label{th:P2GP_finite_convergence}
Let us assume that problem~\eqref{SLBQP} is strictly convex
and $\bx^*$ is its optimal solution. Let  $\left\{\bx^k\right\}$ be a sequence
in $\Omega$ generated by the P2GP method applied to~\eqref{SLBQP}, in which
the minimization phase is performed by any algorithm that is exact for strictly
convex quadratic programming. If one of the following conditions holds:
\begin{itemize}
\item[(i)] $\bx^*$ is  nondegenerate,
\item [(ii)] $\bx^*$ is  degenerate and $\Gamma \geq \kappa(H_q)^{1/2}$,
where $H_q$ is defined in Theorem~\ref{th:Dostal3.2},
\end{itemize}
\vskip 1pt
then $\bx^k=\bx^*$ for $k$ sufficiently large.
\end{theorem}
\begin{proof}
\emph{(i)} By Lemma~\ref{lemma:P2GP_identifies_A1}, in case of nondegeneracy
$\mathcal{A}^k =\mathcal{A}^*$ for $k$ sufficiently large, and the thesis trivially holds.

\emph{(ii)} Thanks to Lemma~\ref{lemma:P2GP_identifies_A1}, we have that P2GP is able
to identify the active nondegenerate variables and the free variables at the solution for
$k$ sufficiently large. This means that there exists $\overline{k}$ such that for $k \geq \overline{k}$
the solution $\bx^*$ of~\eqref{SLBQP} is also solution of \\[-13pt]
\begin{equation} \label{pb_omega_k}
\begin{array}{rl}
\min                    & f(\bx), \\[2pt]
\mbox{\rm s.t.}   & \bx \in \Omega(\bx^k).
\end{array}
\end{equation}
\par \vskip -6pt
\noindent Now assume that $\Gamma \geq \kappa(H_q)^{1/2}$ and
suppose by contradiction that there exists $\widehat{k} \geq \overline{k}$ such that
$\Vert \bbeta (\bx^{\widehat{k}}) \Vert_\infty > \Gamma \Vert \bphi(\bx^{\widehat{k}}) \Vert_2$.
Then, by Theorem~\ref{th:Dostal3.2} it is $\bbeta(\widehat{\bx})\neq 0$, where
$\widehat\bx$ is the solution of~\eqref{pb_omega_k} with $k = \widehat{k}$.
Since $\widehat{\bx}=\bx^*$, this contradicts the optimality of~$\bx^*$.
Therefore, $\bx^k$ is a proportional iterate for $k \ge \widehat{k}$
and P2GP will use the algorithm of the minimization phase to
determine the next iterate. Two cases are possible:
\begin{itemize}[leftmargin=.65cm]
\item[(a)] $\bx^{k+1} = \bx^*$, therefore the thesis holds;
\item[(b)] $\bx^{k+1} \neq \bx^*$ is proportional and such that $\mathcal{A}(\bx^k) \subsetneq 
\mathcal{A}(\bx^{k+1})$, therefore $\bx^{k+2}$ will be computed using again the algorithm
of the minimization phase. Since the active sets are nested, either P2GP is able to find 
$\mathcal{A}^*$ in a finite number of iterations or at a certain iteration it falls in case (a),
and hence the thesis is proved.
\end{itemize}
\vspace*{-3.8mm} \hfill \end{proof}

\vspace*{2pt}
\subsection{Identification phase\label{sec:identification}}

In the identification phase (Steps~\ref{alg:ident_start}-\ref{alg:ident_end} of 
Algorithm~\ref{alg:P2GP}), every projected gradient step needs the computation of
a steplength $\alpha^k$ satisfying the sufficient decrease
condition~\eqref{SuffDecrCond1}-\eqref{SuffDecrCond3}. According to~\cite{More:1989},
this steplength can be obtained by generating a sequence
$\{ \alpha_i^k \}$ of positive trial values such that
\begin{eqnarray}
& & \qquad \qquad \alpha_0^k \in [\gamma_2, \gamma_1] \label{safeguard1} \\
& & \alpha_i^k \in [ \gamma_4 \alpha_{i-1}^k, \gamma_5 \alpha_{i-1}^k], \quad i>0, \label{safeguard2}
\end{eqnarray}
where $\gamma_1$ and  $\gamma_2$ are given in~\eqref{SuffDecrCond2} and
$\gamma_4 < \gamma_5 < 1$,
and by setting $\alpha^k$ to the first trial value that satisfies~\eqref{SuffDecrCond1}.
Note that in practice $\gamma_2$ is a very small value and $\gamma_1$ is a very large one;
therefore, we assume for simplicity that~\eqref{safeguard2} holds for all the choices of
$\alpha_0^k$ described next.

Motivated by the results reported in~\cite{Dai:2005} for BQPs, we compute
$\alpha^k_0$ by using a BB-like rule. 
Following recent studies on steplength selection in gradient
methods~\cite{diSerafino:2016,diSerafino:2017}, we set  $\alpha^k_0$ equal to the
${\rm ABB}_{\rm min}$ steplength proposed in~\cite{Frassoldati:2008}:
\begin{equation} \label{ABBmin}
  \alpha^k_{{\rm ABB}_{\rm min}}=
  \left\{ \!\!
  \begin{array}{ll}
  \min\left\{\alpha^j_{\rm BB2} \, : \, j=\max\{m,k-q\},\ldots,k\right\} & 
       \displaystyle  \mbox{if  } \,  \frac{ \alpha^k_{\rm BB2}}{ \alpha^k_{\rm BB1}} <\tau, \\[3mm]
  \alpha^k_{\rm BB1} & \mbox{otherwise},
  \end{array}
  \right.
\end{equation}
where $m$ is defined in Step~\ref{alg:ident_start} of Algorithm~\ref{alg:P2GP},
$q$ is a nonnegative integer, $\tau\in(0,1)$, and \\[-10pt]
$$
   \alpha^k_{\rm BB1} = \frac{\Vert \bs^{k-1}\Vert^2}{(\bs^{k-1})^T\by^{k-1}}, \qquad
   \alpha^k_{\rm BB2} =  \frac{(\bs^{k-1})^T\by^{k-1}}{\Vert \by^{k-1}\Vert^2},
$$
with ${\bs^{k-1}} =\bx^k -\bx^{k-1}$ and ${\by^{k-1}}=\nabla f^k - \nabla f^{k-1}$.
Details on the rationale behind the criterion used to switch between the BB1 and BB2 steplengths
and its effectiveness are given in~\cite{Frassoldati:2008,diSerafino:2017}.
%
%

If $\alpha^k_0 > 0$, we build the trial steplengths by using a quadratic interpolation
strategy with the safeguard~\eqref{safeguard2} (see, e.g.,~\cite{More:1989}). If $\alpha^k_0 \leq 0$,
we check if $(\nabla_\Omega f^k)^T\,H\,(\nabla_\Omega f^k)\leq 0$, which implies that
the problem
$$
\begin{array}{rl}
\min                    & f(\bx_k+\bv), \\[2pt]
\mbox{\rm s.t.}   & \bq^T\bv=0, \;\;\, v_i=0 \, \mbox{ if } i \in \mathcal{B}^k
\end{array}
$$
is unbounded below along the direction $\nabla_\Omega f^k$. In this case we compute
the breakpoints along $\nabla_{\Omega} f^k$ \cite{More:1989}.
For any $\bx \in \Omega$ and any direction $\bp \in T_{\Omega}(\bx)$,
the breakpoints $\omega_i$, with $i \in \{ j \, : \, p_j \ne 0 \}$, are given by
the following formulas:
\begin{eqnarray*}
& & \mbox{if } p_i<0, \mbox{ then } \omega_i = + \infty \mbox{ if } l_i = -\infty, \mbox{ and } 
\omega_i = \frac{l_i -x_i}{p_i} \mbox{ otherwise}; \\
& & \mbox{if } p_i>0, \mbox{ then } \omega_i = + \infty \mbox{ if } u_i = +\infty, \mbox{ and } 
\omega_i = \frac{u_i -x_i}{p_i} \mbox{ otherwise}.
\end{eqnarray*}
If the minimum breakpoint, which equals $\max \left\lbrace \alpha > 0 \; \mbox{ s.t.} \;\,
\bx^k-\alpha \nabla_\Omega f^k\in\Omega \right\rbrace$, is infinite,
then problem~\eqref{SLBQP} is unbounded. Otherwise, we set 
$ \alpha^k_0 = \bar\omega$,
where $\bar \omega$ is the maximum finite breakpoint.
If $\alpha^k_0$ does not satisfy the sufficient decrease condition,
we reduce it by backtracking until this condition holds.
%
%
Finally, if $\alpha^k_0 \leq 0$ and $(\nabla_\Omega f^k)^T\,H\,(\nabla_\Omega f^k) > 0$,
we set
$$
     \alpha^k_0 = 
     - \frac{(\nabla_\Omega f^k)^T\nabla f^k}{(\nabla_\Omega f^k)^T H\,(\nabla_{\Omega} f^k)},
$$
and proceed by safeguarded quadratic interpolation.\footnote{In Algorithm~\ref{alg:P2GP}
we do not explicitly consider $\alpha^k_0$ in order to simplify the description.}


The identification phase is terminated according to the conditions described at
the beginning of Section~\ref{sec:method}.

\subsection{Minimization phase\label{sec:minimization}}

The minimization phase (Steps~\ref{alg:minim_start}-\ref{alg:minim_end} of
Algorithm \ref{alg:P2GP}) requires the approximate solution of
$$
\begin{array}{rl}
\min &  f(\bx^k + \bd) , \\[1pt]
\mbox{s.t.} & \bq^T\bd=0, \;\;\, d_i=0 \; \mbox{ if } i \in \mathcal{A}(\bx^k) ,
\end{array}
$$
\par \vskip -4pt
\noindent which is equivalent to \\[-14pt]
\begin{equation} \label{subplSLBQP}
\begin{array}{rl}
\min & \displaystyle g (\by) := \frac{1}{2} \, \by^T H_{\mathcal{FF}} \, \by + (\nabla f_{\mathcal{F}}^k)^T \by, \\[4pt]
\mbox{s.t.} & \bq_{\mathcal{F}}^T \, \by=0, \;\;\, \by \in {\Re}^{s},
\end{array}	
\end{equation}
where $\mathcal{F} = \mathcal{F}^k$ and $s = \vert \mathcal{F} \vert$. 

Problem~\eqref{subplSLBQP} can be 
formulated as an unconstrained quadratic minimization problem by using a Householder
transformation
$$
P = I - \bw \bw^T \in \Re^{s \times s}, \quad \Vert \bw\Vert=\sqrt 2, 
\quad P \bq_{\mathcal{F}} = \sigma \be_1,
$$
where $\sigma = \pm \| \bq_{\mathcal{F}} \|$ (see, e.g., \cite{Bjorck:1996}).
Letting 
$\by = P\bz$, $M = P H_{\mathcal{FF}} P$ and $\br = P \nabla f_{\mathcal{F}}^k$, 
problem~\eqref{subplSLBQP} becomes \\[-16pt]
$$
\begin{array}{rl}
\min & p(\bz) := \displaystyle \frac{1}{2} \, \bz^T M \, \bz + \br^T \bz , \\[2pt]
\mbox{s.t.} & z_1 = 0,
\end{array}	
$$
\par \vskip -4pt
\noindent which simplifies to \\[-14pt]
\begin{equation} \label{defsubp1unc}
   \min_{\widetilde{\bz} \in \Re^{s-1}} \widetilde{p}(\widetilde{\bz}) :=
           \frac{1}{2} \, \widetilde{\bz}^T \widetilde{M} \, \widetilde{\bz} +
          \widetilde{\br}^T \widetilde{\bz},
\end{equation}
\par \vskip -4pt
\noindent where \\[-12pt]
$$
M = \left(  \begin{array}{cc} m_{11} & \widetilde{\bm}^T  \\ 
\widetilde{\bm} & \widetilde{M} \end{array} \right) \!, 
\quad
\br = \left( \begin{array}{c} r_1\\
\widetilde{\br} \end{array} \right) \!,
\quad
\bz = \left( \begin{array}{c} z_1\\
\widetilde{\bz} \end{array} \right) \!.
$$

We note that $\bq_\mathcal{F} = \sigma P \be_1$ , i.e.,  $\bq_\mathcal{F}$
is a multiple of the first column of $P$, and hence the remaining columns of
$P$ span $\{ \bq_\mathcal{F} \}^\perp$. Furthermore, a simple computation shows that
$\widetilde{M} = \widetilde{P}^T H_{\mathcal{F}\mathcal{F}} \widetilde{P}$, where
$\widetilde{P}$ is the matrix obtained by deleting the first column of $P$.
By reasoning as in the proof of Theorem~\ref{th:Dostal3.2} (see~\eqref{zetamin}),
we find that $\zeta_{min}(\widetilde{M}) \ge \zeta_{min}(H_q)$ and
$\zeta_{max}(\widetilde{M}) \le \zeta_{max}(H_q)$, 
where $H_q = V^T H V$ and $V \in \Re^{n \times (n-1)}$ is any matrix with orthogonal columns
spanning $\{ \bq \}^\perp$. Therefore, if $H_q$ is positive definite, then \\[-6mm]
$$
\kappa (\widetilde{M}) \le \kappa (H_q).
$$
For any other $Z \in \Re^{n \times (n-1)}$ with orthogonal columns
spanning $\{ \bq \}^\perp$, we can write $V^T = D Z^T$ with $D \in \Re^{(n-1) \times(n-1)}$ orthogonal;
therefore, $V^T H V$ and $Z^T H Z$ are similar and $\kappa (H_q)$ does not depend
on the choice of the orthonormal basis of $\{ \bq \}^\perp$. Furthermore,
if $H$  is positive definite, by the Cauchy's interlace theorem \cite[Theorem~10.1.1]{Parlett:1998}
it is $\kappa(H_q) \leq \kappa(H)$.

The finite convergence results presented in Section~\ref{sec:method} for strictly convex
problems rely on the exact solution of~\eqref{defsubp1unc}. In infinite precision,
this can be achieved by means of the CG algorithm, as in the GPCG method. Of course,
in presence of roundoff errors, finite convergence is generally neither obtained nor required.

We can solve~\eqref{defsubp1unc} by efficient gradient methods too. In this work, we investigate the use of the SDC
gradient method~\cite{Deasmundis:2014} as a solver for the minimization phase
in the strictly convex case. The SDC method uses the following steplength:
\begin{equation} \label{alpha_sdc}
\alpha^k_{\rm SDC}= \left\{ \!\!
\begin{array}{ll}
    \alpha^k_{\rm C} & \textrm{if mod}\left( k, \bar{k} +l \right) < \bar{k}, \\[2pt]
    \alpha^t_{Y} & \textrm{otherwise, with} \; t \; =\textrm{max}\{i\leq k  \, : \,
            \textrm{mod}\left( i, \bar{k} +l  \right)=\bar{k}\},
\end{array}
\right.
\end{equation}
where $ \bar{k} \ge 2$, $l \ge 1$, $\alpha^k_{\rm C}$ is the Cauchy steplength and
\begin{equation} \label{yuan_step}
   \alpha^t_{\rm Y} = 2 \left( \sqrt{ \left( \frac{1}{\alpha^{t-1}_{\rm SD}} - \frac{1}{\alpha^t_{\rm SD}} \right)^2
   + 4\frac{\Vert \nabla f^t \Vert^2} {\left( \alpha^{t-1}_{\rm SD} \Vert \nabla f^{t-1} \Vert \right)^2 } }+
   \frac{1}{\alpha^{t-1}_{\rm SD}}+\frac{1}{\alpha^t_{\rm SD}} \right)^{-1}
\end{equation}
is the Yuan  steplength~\cite{Yuan:2006}.
The interest for this steplength is motivated by its spectral properties, which dramatically speed up
the convergence~\cite{Deasmundis:2014,diSerafino:2017}, while showing
certain regularization properties useful to deal with linear ill-posed
problems~\cite{Deasmundis:2016}. Similar properties hold for the SDA gradient
method~\cite{Deasmundis:2013}, but for the sake of space we do not show the results of
its application in the minimization phase. It is our opinion that the P2GP framework provides
also a way to exploit these methods when solving linear ill-posed problems with bounds
and a single linear constraint. 

Once a descent direction $\bd^k$ is obtained by using CG
or SDC,  a full step along this direction is performed starting from $\bx^k$,
and $\bx^{k+1}$ is set equal to the resulting point if this is feasible. Otherwise
$\bx^{k+1} = P_{\Omega^k}(\bx^k + \alpha^k \bd^k)$ where $\alpha^k$ satisfying the sufficient
decrease conditions is computed by using safeguarded quadratic interpolation~\cite{More:1991}.

If the problem is not strictly convex, we choose the CG method for the minimization phase. 
If CG finds a direction $\bd^k$ such that $ (\bd^k)^T H \, \bd^k \leq 0$ we set
$\bx^{k+1} = \bx^k + \alpha^k \bd^k$, where $\alpha^k$ is the largest feasible steplength,
i.e., the minimum breakpoint along $\bd^k$,
unless the objective function results to be unbounded along $\bd^k$.

As already observed, the stopping criterion in the solution of problem~\eqref{defsubp1unc}
must not be too stringent, since the decision of continuing the minimization on the
reduced space is left to the proportionality criterion. In order to stop the solver for
problem~\eqref{defsubp1unc}, we check the progress in the reduction of the
objective function as in the identification phase, i.e., we terminate
the iterations if
\begin{equation} \label{stopCG}
\widetilde{p}(\widetilde{\bz}^j)- \widetilde{p}(\widetilde{\bz}^{j+1})  \le \xi\, 
\max_{1 \le l < j} \left\{ \widetilde{p}(\widetilde{\bz}^l)  - \widetilde{p}(\widetilde{\bz}^{l+1}) \right\} ,
\end{equation}
where $\xi \in (0,1)$ is not too small (the value used in the numerical experiments is given in
Section~\ref{sec:experiments}). This choice follows~\cite{More:1991}. If the active set
has not changed and the current iterate is proportional, the minimization phase does not restart
from scratch, but the minimization method continues its iterations as it had not been stopped.

\subsection{Projections\label{sec:projection}}

P2GP requires projections onto $\Omega$ (Step~\ref{alg:proj1}  of Algorithm~\ref{alg:P2GP}),
onto $\Omega^k = \Omega\,\cap\,\Omega(\bx^k)$ (Step~\ref{alg:proj2} of Algorithm~\ref{alg:P2GP}),
and onto $T_\Omega(\bx^k)$ (for the computation of $\bbeta(\bx^k)$).
%
%
We perform the projections by using the algorithm proposed by Dai and Fletcher in~\cite{Dai:2006a}.

\section{Numerical experiments\label{sec:experiments}}

In order to analyze the behavior of P2GP using both CG and SDC in the minimization phase,
we performed numerical experiments on several problems, either generated with the aim of
building test cases with varying characteristics (see Section~\ref{sec:synthpbs}) or coming from
SVM training (see Section~\ref{sec:svmpbs}).

On the first set of problems, referred to as random problems because of the way they are
built, we compared both versions of P2GP with the following methods:
\begin{itemize}
\item 
GPCG-like, a modification of P2GP where the termination of the minimization phase
(performed by CG) is not driven by the proportionality criterion, but by the bindingness
of the active variables, like in the GPCG method;
\item
PABB$_{\rm min}$, a Projected Alternate BB method executing the line search as in P2GP
and computing the first trial steplength with the ${\rm ABB}_{\rm min}$ rule described in
Section~\ref{sec:identification};
\end{itemize}
The first method was selected to evaluate the effect of the proportionality-based criterion in the
minimization phase, the second one because of its effectiveness among
general GP methods. P2GP,  GPCG-like, and PABB$_{\rm min}$ were implemented in Matlab.

To further assess the behavior of P2GP, we also compared it, on the random and
SVM problems, with the GP method implemented in BLG, a C code available from
\url{http://users.clas.ufl.edu/hager/papers/Software/}. BLG solves
nonlinear optimization problems with bounds and a single linear constraint,
and can be considered as a benchmark for software based on gradient methods.
Its details are described in~\cite{Hager:2006,GonzalezLima:2011}.

The following setting of the parameters was considered for P2GP:
$\eta = 0.1$ in~\eqref{stopGP} and $\xi = 0.5$ in~\eqref{stopCG};
$\mu_1 = 10^{-4}$ in~\eqref{SuffDecrCond1};
$\gamma_1=10^{12}$, $\gamma_2=10^{-12}$, $\gamma_3=10^{-2}$, and $\gamma_4=0.5$ 
in~\eqref{safeguard1}-\ref{safeguard2}; $q=3$ and $\tau=0.2$ in~\eqref{ABBmin}.
Furthermore, when SDC was used in the minimization phase, $\bar{k} = 6$ and $l=4$ were chosen
in~\eqref{alpha_sdc}. A maximum number of 50 consecutive GP and CG (or SDC) iterations was
also considered. The previous choices were also used for
the GPCG-like method, except for the parameter $\xi$, which was set to 0.25. The parameters of PABB$_{\rm min}$
in common with P2GP were given the same values too, except $\tau$, which was computed by the adaptive
procedure described in~\cite{Bonettini:2009}, with 0.5 as starting value. Details on the stopping conditions
used by the methods are given in Sections~\ref{sec:synthresults} and \ref{sec:svmresults}, where the results
obtained on the test problems are discussed.

About the proportionality condition~\eqref{proportioning},
a conservative approach would suggest to adopt a large value for $\Gamma$.
However, such a choice is likely to be unsatisfactory in practice; in fact,
a large $\Gamma$ would foster high accuracy in
the minimization phase, even at the initial steps of the
algorithm, when the active constraints at the solution are far from
being identified. Thus, we used the following adaptive strategy
for updating $\Gamma$ after line 37 of Algorithm~\ref{alg:P2GP}:
\smallskip
\begin{center}
  \begin{algorithmic}
	\If {$\Vert \bbeta^{k} \Vert_\infty > \Gamma \,\Vert \bphi^{k} \Vert_2$ }
	   \State {$ \Gamma = \max \left\lbrace 1.1\cdot\Gamma ,\;1  \right\rbrace;$}
	\ElsIf {$\mathcal{A}^k \neq \mathcal{A}^{k-1}$}
	   \State {$ \Gamma = \max \left\lbrace 0.9\cdot\Gamma ,\;1  \right\rbrace;$}
	\EndIf
  \end{algorithmic}
\end{center}
\par
\vskip 1pt
\noindent
Based on our numerical experience, we set the starting value of $\Gamma$ equal to 1.

BLG was run using the gradient projection search direction
(it also provides the Frank-Wolfe and affine-scaling directions). However, the code could
switch to the Frank-Wolfe direction, according to inner automatic criteria. Note that BLG uses a cyclic BB
steplength $\bar\alpha^k$ as trial steplength, together with an adaptive nonmonotone line search along the
feasible direction $P_{\Omega}(\bx^k - \bar\alpha^k \nabla f^k) - \bx^k$ (see \cite{Hager:2006}
for the details). Of course, the BLG features exploiting the form of a quadratic objective function were used.
The stopping criteria applied with the random problems and the SVM ones are specified
in Sections~\ref{sec:synthresults} and \ref{sec:svmresults}, respectively. Further details on the 
use of BLG are given there.

All the experiments were carried out using a 64-bit Intel Core i7-6500, with maximum clock frequency
of 3.10 GHz, 8 GB of RAM, and 4 MB of cache memory. BLG (v.~1.4) and SVMsubspace (v.~1.0)
were compiled by using gcc 5.4.0. P2GP, GPCG-like, and PABB$_{\rm min}$ were run under MATLAB 7.14 (R2012a).
The elapsed times reported for the Matlab codes were measured by using the \texttt{tic} and
\texttt{toc} commands.

\subsection{Random test problems\label{sec:synthpbs}}

The implementations of all methods were run on random SLBQPs built
by modifying the procedure for generating BQPs proposed in~\cite{More:1989}.
The new procedure first computes a point $\bx^*$
and then builds a problem of type~\eqref{SLBQP} having $\bx^*$ as stationary point.
Obviously, if the problem is strictly convex, $\bx^*$ is its solution. The
following parameters are used to define the problem:
\begin{itemize}
\item \texttt{n}, number of variables (i.e., $n$);
\item \texttt{ncond}, $\log_{10} \kappa (H)$;
\item \texttt{zeroeig} $ \in[0, 1)$, fraction of zero eigenvalues of $H$;
\item \texttt{negeig} $ \in[0, 1)$, fraction of negative eigenvalues of $H$;
\item \texttt{naxsol} $ \in[0, 1)$, fraction of active variables at $\bx^*$;
\item \texttt{degvar} $\in[0,1)$, fraction of active variables at $\bx^*$ that are degenerate; 
\item \texttt{ndeg} $\in\lbrace0,1,2,\ldots\rbrace$, amount of near-degeneracy;
\item \texttt{linear}, 1 for SLBQPs, and 0 for BQPs;
\item \texttt{nax0} $\in[0,1)$, fraction of active variables at the starting point.
\end{itemize}

The components of $\bx^*$ are computed as random numbers from the uniform
distribution in $(-1,1)$.
All random numbers considered next are from uniform distributions too.
The Hessian matrix $H$  is defined as \\[-14pt]
\begin{equation}
\label{eq:randhess}
H = G\,D\,G^T, 
\end{equation}
\par \vskip -3pt
\noindent where $D$ is a diagonal matrix and
$G = ( I - 2 \,\bp_3 \bp_3^T) (I - 2\,\bp_2 \bp_2^T) ( I - 2\,\bp_1 \bp_1^T)$,
with $\bp_j$ unit vectors. For $j=1,2,3$, the components of $\bp_j$ are obtained by
generating $\bar \bp_j = (\bar p_{ji} )_{i=1,\ldots,n}$, where the values $\bar p_{ji}$ are
random numbers in $(-1,1)$, and setting $\bp_j = \bar \bp_j / \| \bar \bp_j \|$. 
The diagonal entries of $D$ are defined as follows:
$$
   d_{ii} = \left\{ \begin{array}{ll} 
     \;\;\; 0    & \mbox{for approximately }  \mathtt{zeroeig} * \mathtt{n} \mbox{ values of } i, \\
     - 10^{\frac{i-1}{n-1}(\mathtt{ncond})} & \mbox{for approximately } 
                                         \mathtt{negeig} * \mathtt{n} \mbox{ values of } i, \label{dneg} \\
     \;\;\; 10^{\frac{i-1}{n-1}(\mathtt{ncond})} & \mbox{for the remaing values of } i .
   \end{array} \right.
$$
We note that \texttt{zeroeig} and \texttt{negeig} are not the actual fraction of zero and negative
eigenvalues. The actual fraction of zero eigenvalues is determined by generating
a random number $\xi_i \in (0,1)$ for each $i$, and by setting $d_{ii} = 0$ if 
$\xi_i \le \mathtt{zeroeig}$; the same strategy is used to determine the actual number
of negative eigenvalues. We also observe that $\kappa (H) = 10^{\,\mathtt{ncond}}$, 
if $H$ has no zero eigenvalues.

In order to define the active variables at $\bx^*$, $n$ random
numbers $\chi _i\in (0,1)$ are computed, and the index $i$ is put in 
$\mathcal{A}^*$ if $\chi_i \leq \mathtt{naxsol}$; then $\mathcal{A}^*$~is partitioned into the sets
$\mathcal{A}^*_N$ and $\mathcal{A}^* \setminus \mathcal{A}^*_N$,
with $| \mathcal{A}^* \setminus \mathcal{A}^*_N |$ approximately equal to
$\lfloor \mathtt{degvar} *  \mathtt{naxsol} *  \mathtt{n} \rfloor$. More precisely,
an index $i$ is put in $\mathcal{A}^* \setminus \mathcal{A}^*_N$
if $\psi_i \le \mathtt{degvar}$, where $\psi_i$ is a random number in $(0,1)$,
and is put in $\mathcal{A}^*_N$ otherwise. The vector $\blambda^*$
of Lagrange multipliers associated with the box constraints at $\bx^*$ is
initially set as
\begin{equation*}
   \lambda_i^* = 
   \left\lbrace\begin{array}{ll}
        10^{-\mu_i\,\mathtt{ndeg}} & \mbox{if } \, i \in \mathcal{A}_N^*,\\
	0 & \mbox{otherwise},
    \end{array}\right.
\end{equation*}
where $\mu_i$ is a random number in $(0,1)$.  Note that the larger \texttt{ndeg},
the closer to 0 is the value of $\lambda_i^*$, for $i \in \mathcal{A}_N^*$ (in this sense
\texttt{ndeg} indicates the amount of near-degeneracy).
The set $\mathcal{A}^*$ is splitted into $\mathcal{A}_l^*$ and $\mathcal{A}_u^*$ as follows:
for each $i \in \mathcal{A}^*$, a random number $\nu_i \in (0,1)$ is generated;
$i$ is put in $\mathcal{A}_l^*$ if $\nu_i < 0.5$, and in $\mathcal{A}_u^*$ otherwise.
Then, if $i \in \mathcal{A}_u^*$, the corresponding Lagrange multiplier is modified by setting
$\lambda_i^* = - \lambda_i^*$. The lower and upper bounds $\bl$ and $\bu$ are defined as
follows:
\vspace*{-1mm}
$$
\begin{array}{llll}
  l_i = -1     & \!\!\! \mbox{and} & \!\!\! u_i = 1       & \!\!\!\! \mbox{if } i \notin \mathcal{A}^*,\\
  l_i = x^*_i & \!\!\! \mbox{and} & \!\!\! u_i = 1       & \!\!\!\! \mbox{if } i \in \mathcal{A}_l^*, \\
  l_i = -1     & \!\!\! \mbox{and} & \!\!\! u_i = x^*_i & \!\!\!\!\mbox{if } i \in \mathcal{A}_u^*. 
\end{array}
$$

\vspace*{-1mm}
If $\mathtt{linear}=0$, the linear constraint is neglected. If $\mathtt{linear} = 1$, the
vector $\bq$ in~\eqref{SLBQP} is computed by randomly generating its components in
$(-1, 1)$, the scalar $b$ is set to $\bq^T \bx^*$, and the vector $\bc$ is defined
so that the KKT conditions at the solution are satisfied: \\[-16pt]
$$
   \bc = \left\lbrace \begin{array}{ll}
      H\,\bx^* - \blambda^*                         & \mbox{if } \mathtt{linear} = 0,\\
      H\,\bx^* - \blambda^* - \rho^* \,\bq  & \mbox{if } \mathtt{linear} = 1,
\end{array} \right.
$$
where $\rho^*$ is a random number in $(-1,1) \setminus \lbrace 0 \rbrace$
representing the Lagrange multiplier associated with the linear constraint.

By reasoning as with $\bx^*$, approximately $\mathtt{nax0}*\mathtt{n}$ components
of the starting point $\bx^0$ are set as $x_i^0 = l_i$ or $x_i^0 = u_i$.
The remaining components are defined as $x_i^0 = (l_i + u_i)/2$. Note that $\bx^0$
may not be feasible; in any case, it will be projected onto $\Omega$ by the optimization
methods considered here.

Finally, we note that although $\bx^*$ is a stationary point of
the problem generated by the procedure described so far,
there is no guarantee that P2GP converges to $\bx^*$
if the problem is not strictly convex.

The following sets of test problems, with size \texttt{n} $= 20000$, were generated:
\begin{itemize}
\item 
27 strictly convex SLBQPs with nondegenerate solutions, obtained by setting
\texttt{ncond} $= 4,5,6$, \texttt{zeroeig} $= 0$, \texttt{negeig} $=0$,
\texttt{naxsol} $= 0.1, 0.5, 0.9$,  \texttt{degvar} $= 0$, \texttt{ndeg} $= 0, 1, 3$, 
and \texttt{linear} $= 1$;
\item 18 strictly convex SLBQPs with degenerate solutions, obtained by setting
\texttt{ncond} $= 4,5,6$, \texttt{zeroeig} $= 0$, \texttt{negeig} $=0$,
\texttt{naxsol} $= 0.1, 0.5, 0.9$,  \texttt{degvar} $= 0.2, 0.5$, \texttt{ndeg} $= 1$, 
and \texttt{linear} $= 1$; 
\item 27 convex (but not stricltly convex) SLBQPs, obtained by setting
\texttt{ncond} $= 4,5,6$, \texttt{zeroeig} $= 0.1, 0.2, 0.5$, \texttt{negeig} $= 0$,
\texttt{naxsol} $= 0.1, 0.5, 0.9$,  \texttt{degvar} $= 0$, \texttt{ndeg} $= 1$, 
and \texttt{linear} $= 1$;
\item 27 nonconvex SLBQPs, obtained by setting
\texttt{ncond} $= 4,5,6$, \texttt{zeroeig} $= 0$, \texttt{negeig} $= 0.1, 0.2, 0.5$,
\texttt{naxsol} $= 0.1, 0.5, 0.9$,  \texttt{degvar} $= 0$, \texttt{ndeg} $= 1$, 
and \texttt{linear} $= 1$;
\end{itemize}
Since BQPs are special cases of SLBQPs, four sets of BQPs were also generated,
by setting \texttt{linear} $= 0$ and choosing all remaining parameters as specified above.
All the methods were applied to each problem with four starting points,
corresponding to \texttt{nax0} $= 0, 0.1, 0.5, 0.9$.

\subsection{SVM test problems\label{sec:svmpbs}}

SLBQP test problems corresponding to the dual formulation of two-class C-SVM
classification problems were also used (see, e.g.,~\cite{Vapnik:1982}). Ten problems from the LIBSVM data set,
available from \url{https://www.csie.ntu.edu.tw/~cjlin/libsvmtools/datasets/}, were considered, whose
details (size of the problem, features and nonzeros in the data) are given in Table~\ref{tab:svmpbs}.
A linear kernel was used, leading to problems with positive semidefinite Hessian matrices.
The penalty parameter C was set to 10.
For most of the problems, the number of nonzeros is much smaller
than the product between size and features, showing that the data are relatively sparse.

\begin{table}[h!]
	\centering
	{\small
	\begin{tabular}{|l|r|r|r|}
	         \hline
		problem  & {size} & {features} & {nonzeros} \\ \hline   
		a6a      &  11220 &        122 &     155608 \\             
		a7a      &  16100 &        122 &     223304 \\             
		a8a      &  22696 &        123 &     314815 \\             
		a9a      &  32561 &        123 &     451592 \\             
		ijcnn1   &  49990 &         22 &     649870 \\              
		phishing &  11055 &         68 &     331650 \\             
		real-sim &  72309 &      20958 &    3709083 \\             
		w6a      &  17188 &        300 &     200470 \\             
		w7a      &  24692 &        300 &     288148 \\             
		w8a      &  49749 &        300 &     579586 \\             
		\hline
	\end{tabular}
	\par \bigskip
	\caption{Details of the SVM test set. \label{tab:svmpbs}}
	}
\end{table}

\subsection{Results on random problems\label{sec:synthresults}}

We first discuss the results obtained by running the implementations
of the P2GP, PABB$_{\rm min}$ and GPCG-like methods on
the problems described in Section~\ref{sec:synthpbs}.
In the stopping condition~\eqref{accur_req}, $tol = 10^{-6} \| \bphi^0 + \bbeta^0 \|$
was used; furthermore, at most $30000$ matrix-vector products and
$30000$ projections were allowed, declaring failures
if these limits were achieved without satisfying condition~\eqref{accur_req}.
The methods were compared by using the performance profiles proposed by Dolan and
Mor\'e~\cite{DolanMore:2002}. We note that the performance profiles in this section
may show a number of failures larger than the actual one,
because the range on the horizontal axis has been limited
to enhance readability. However, all the failures will be explicitly reported in the text.

\begin{figure}[t]
   \centering
   \setlength{\tabcolsep}{-8pt}
   \begin{tabular}{cc}
        \hspace*{-9pt}
        \includegraphics[width=0.55\textwidth]{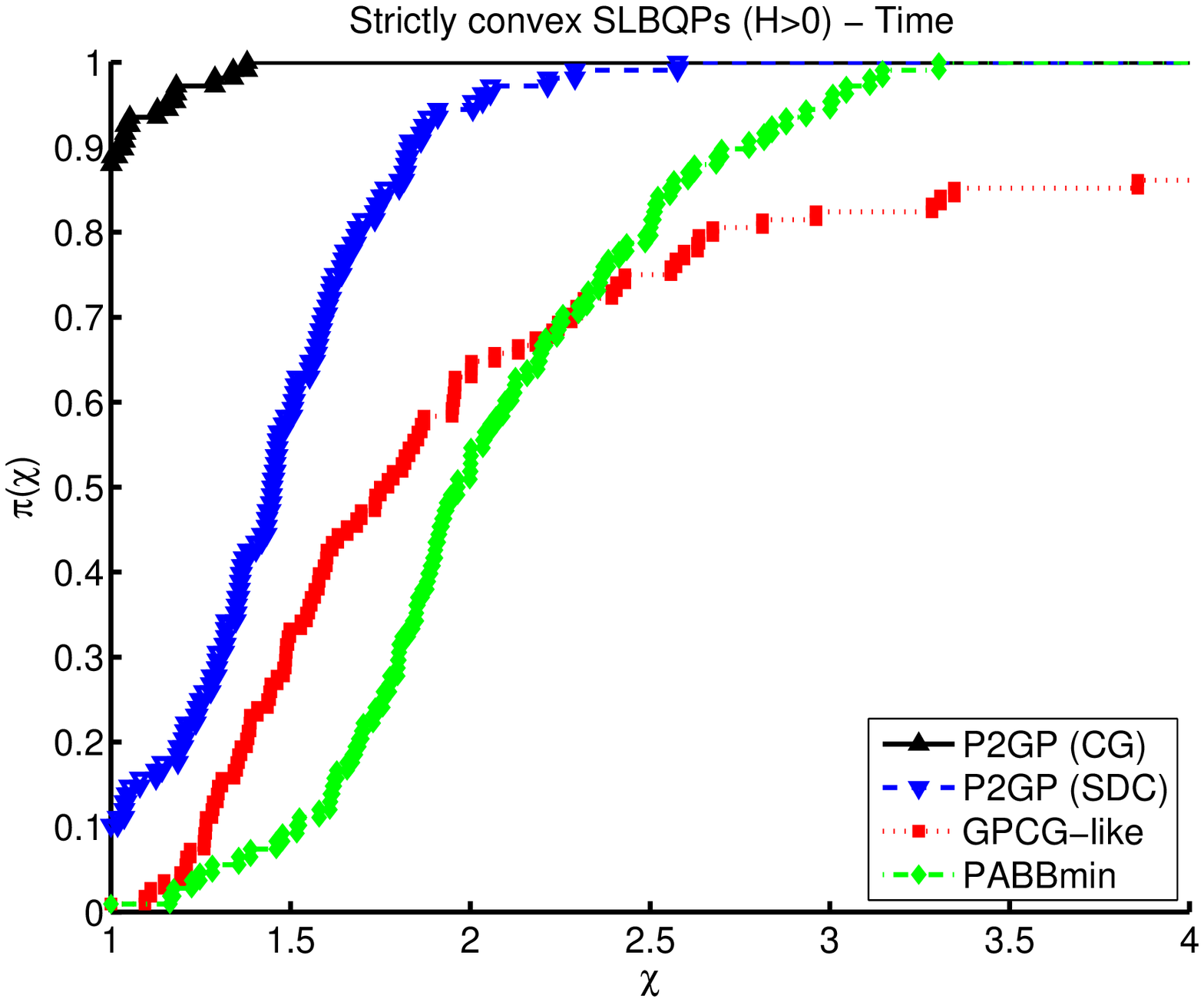} &
        \includegraphics[width=0.55\textwidth]{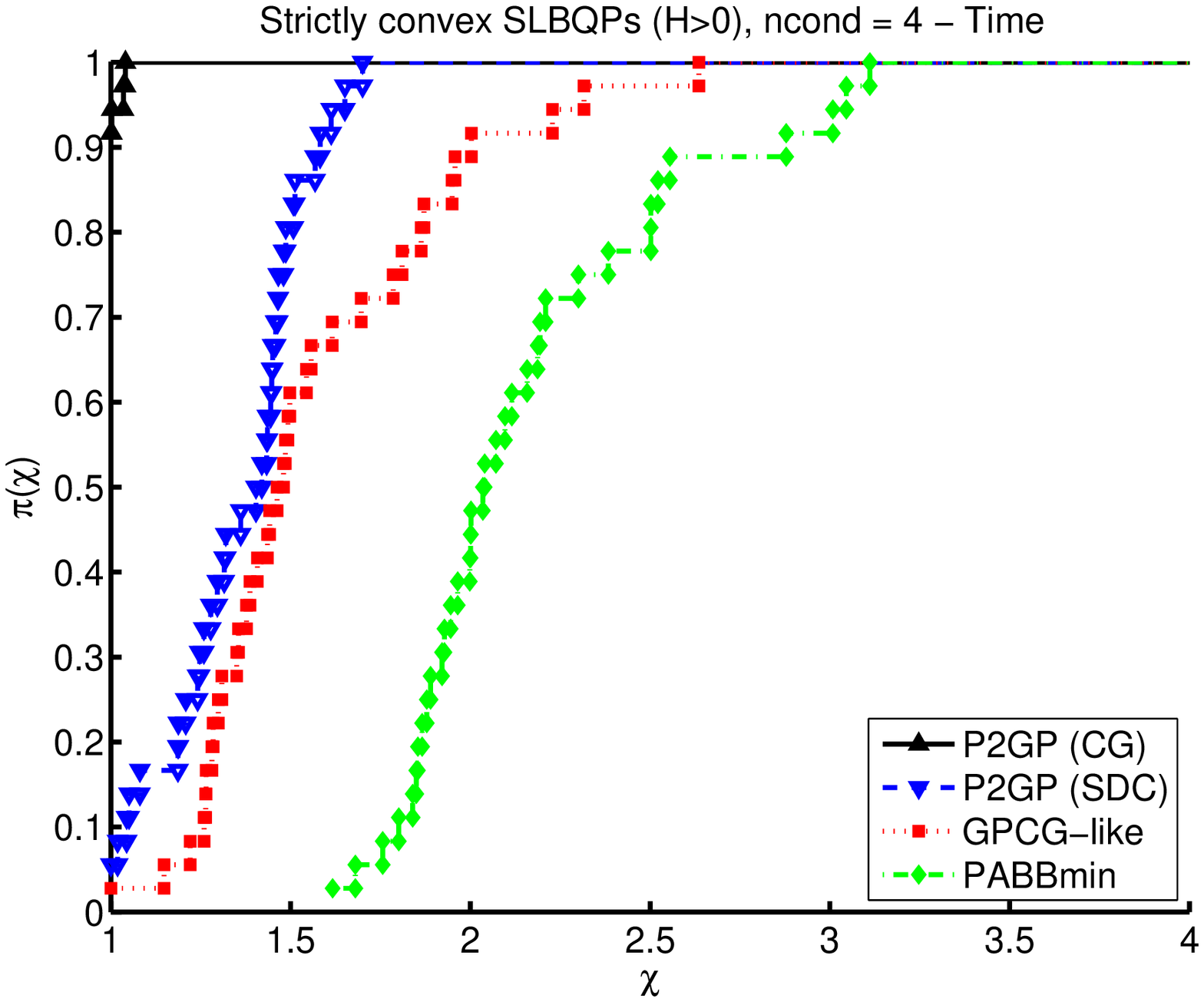} \\[-4pt]
        \hspace*{-9pt}
        \includegraphics[width=0.55\textwidth]{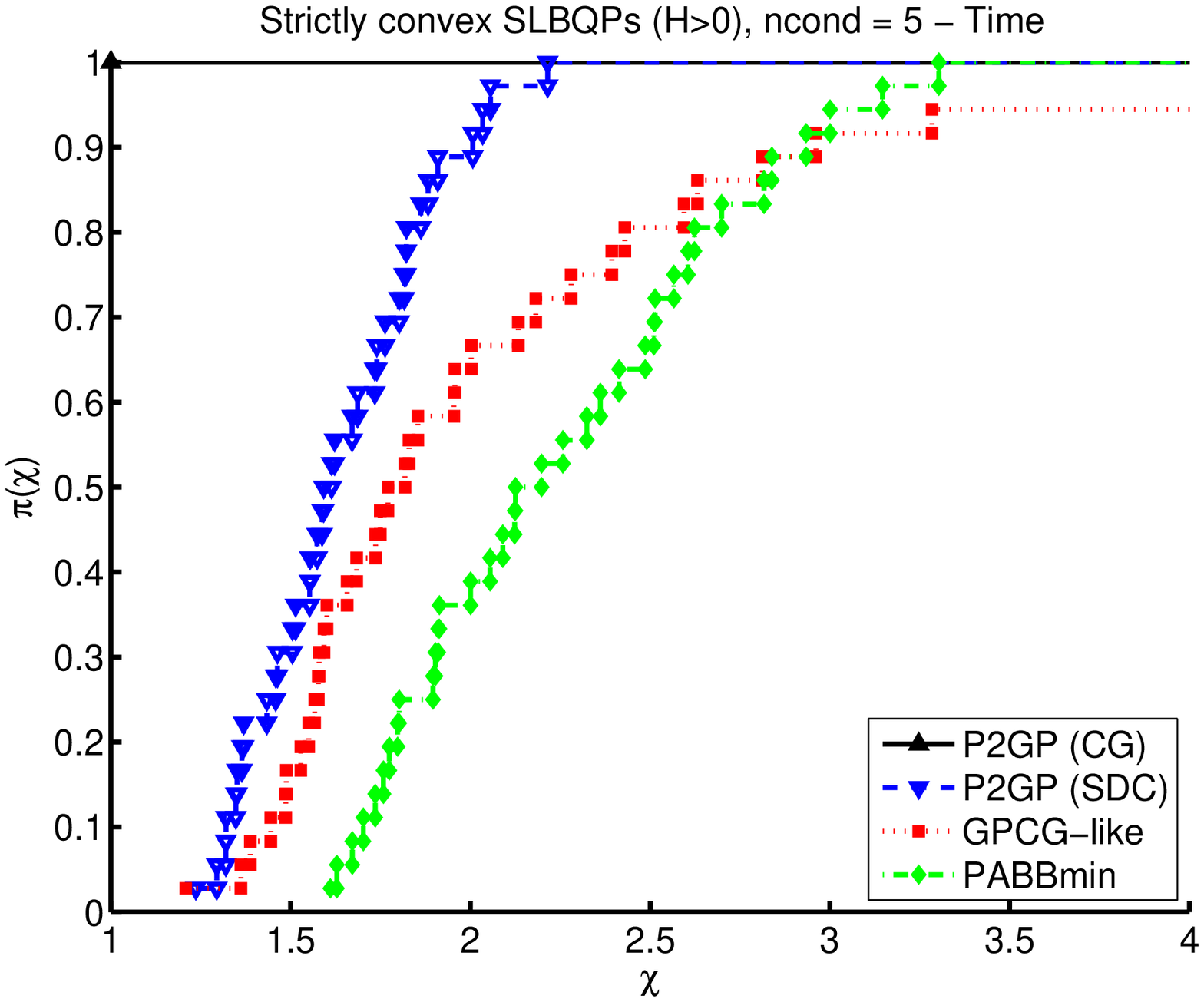}  &
	\includegraphics[width=0.55\textwidth]{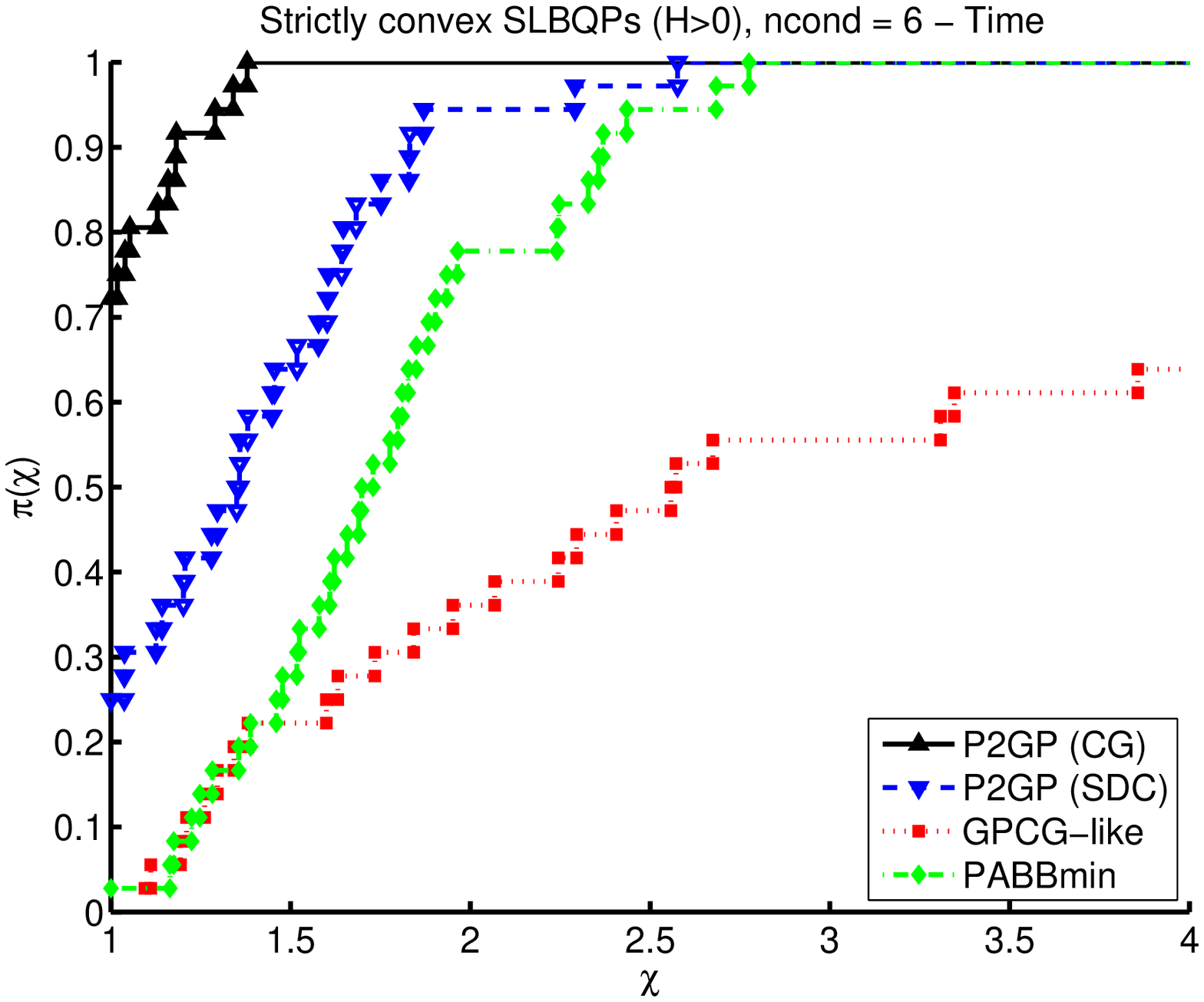} 
   \end{tabular}
   \vskip -10pt
   \caption{Performance profiles of P2GP with CG and SDC, PABB$_{\rm min}$,
and GPCG-like on strictly-convex SLBQPs with nondegenerate solutions:
execution times for all the problems (top left), for $\kappa(H) = 10^4$ (top right),
for $\kappa(H) = 10^5$ (bottom left), and for $\kappa(H) = 10^6$ (bottom right).
\label{fig:sconv_time}}
\end{figure}
\begin{figure}[t]
   \centering
      \setlength{\tabcolsep}{-8pt}
      \begin{tabular}{cc}
         \hspace*{-9pt}
         \includegraphics[width=0.55\textwidth]{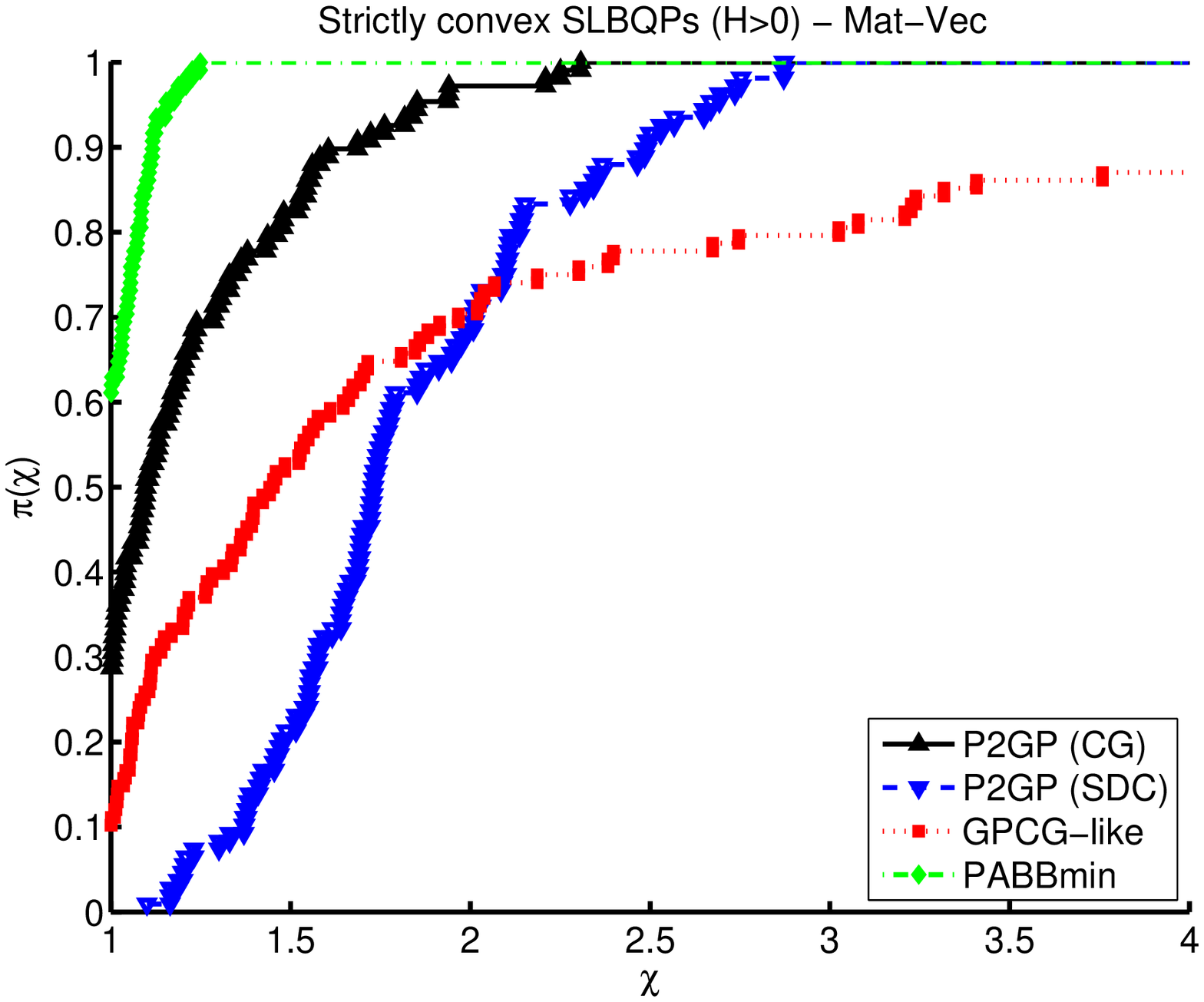} &
        \includegraphics[width=0.55\textwidth]{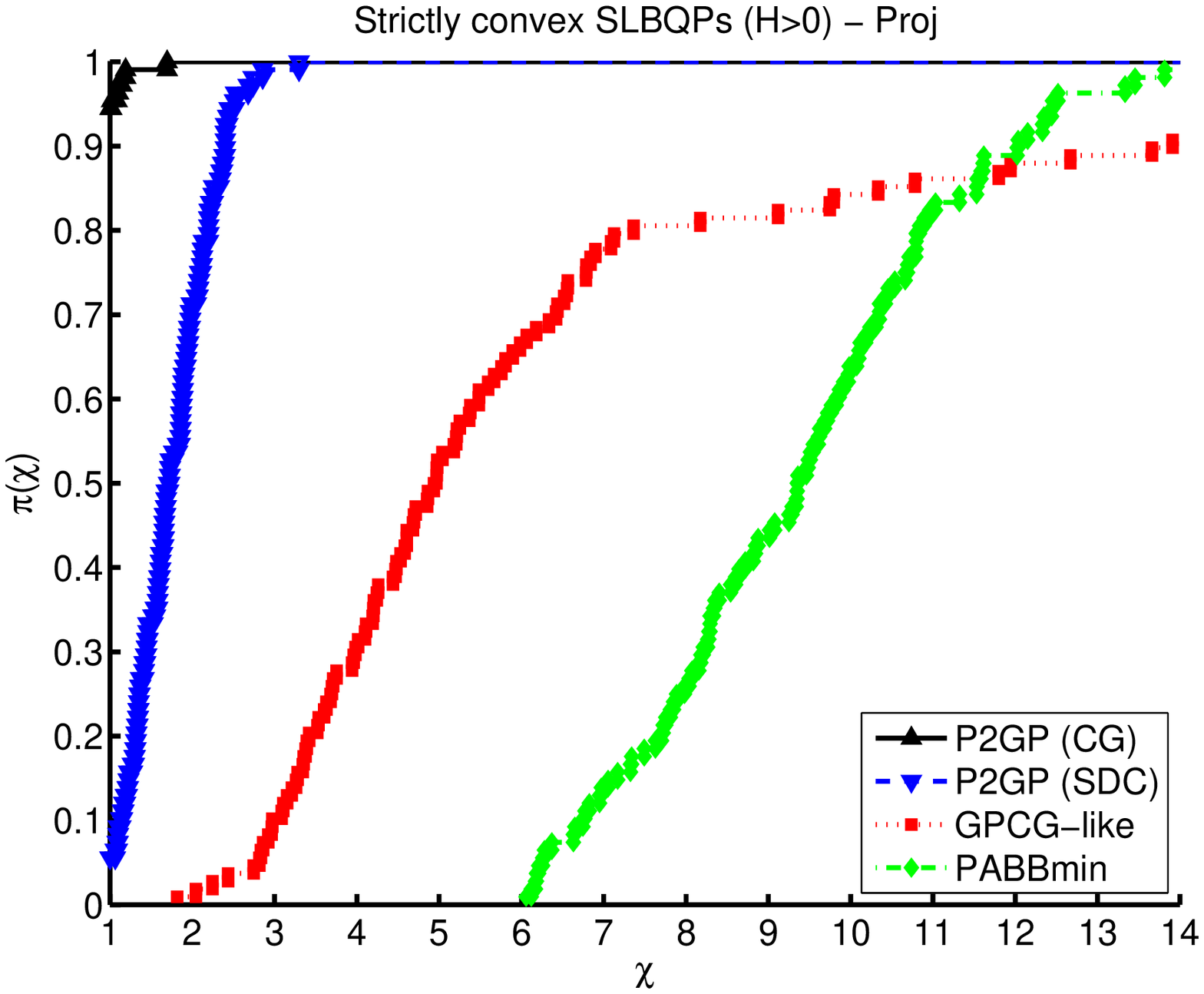}
   \end{tabular}
    \vskip -10pt
   \caption{Performance profiles of P2GP with CG and SDC, PABB$_{\rm min}$,
and GPCG-like on strictly convex SLBQPs with nondegenerate solutions:
number of matrix-vector products (left) and projections (right).
\label{fig:sconv_mv_p}}
\end{figure}

\begin{figure}[h!]
   \centering
   \hspace*{-9pt}
   \includegraphics[width=0.55\textwidth]{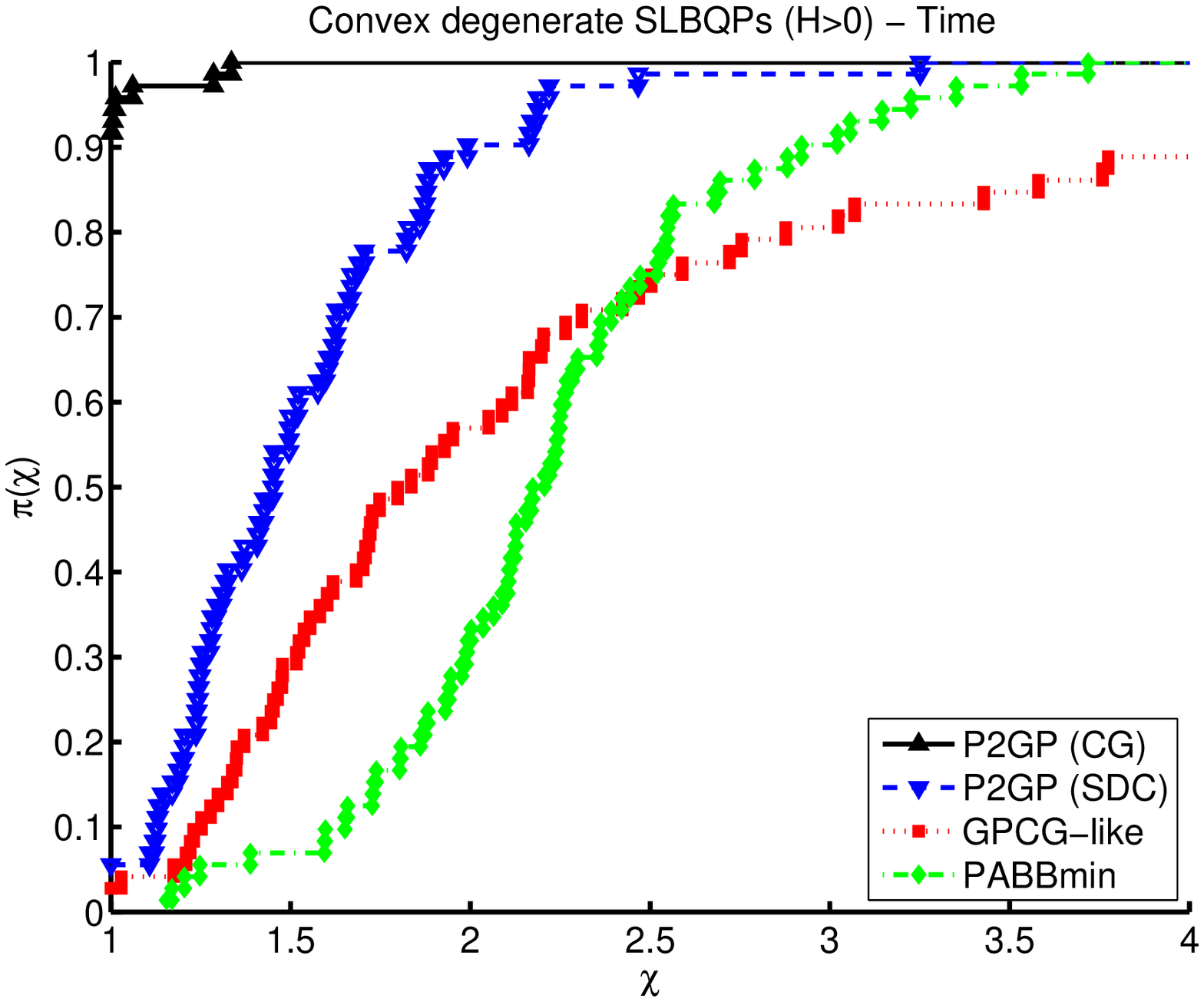} \\[-4pt]
   \setlength{\tabcolsep}{-8pt}
   \begin{tabular}{cc}
      \hspace*{-9pt}
      \includegraphics[width=0.55\textwidth]{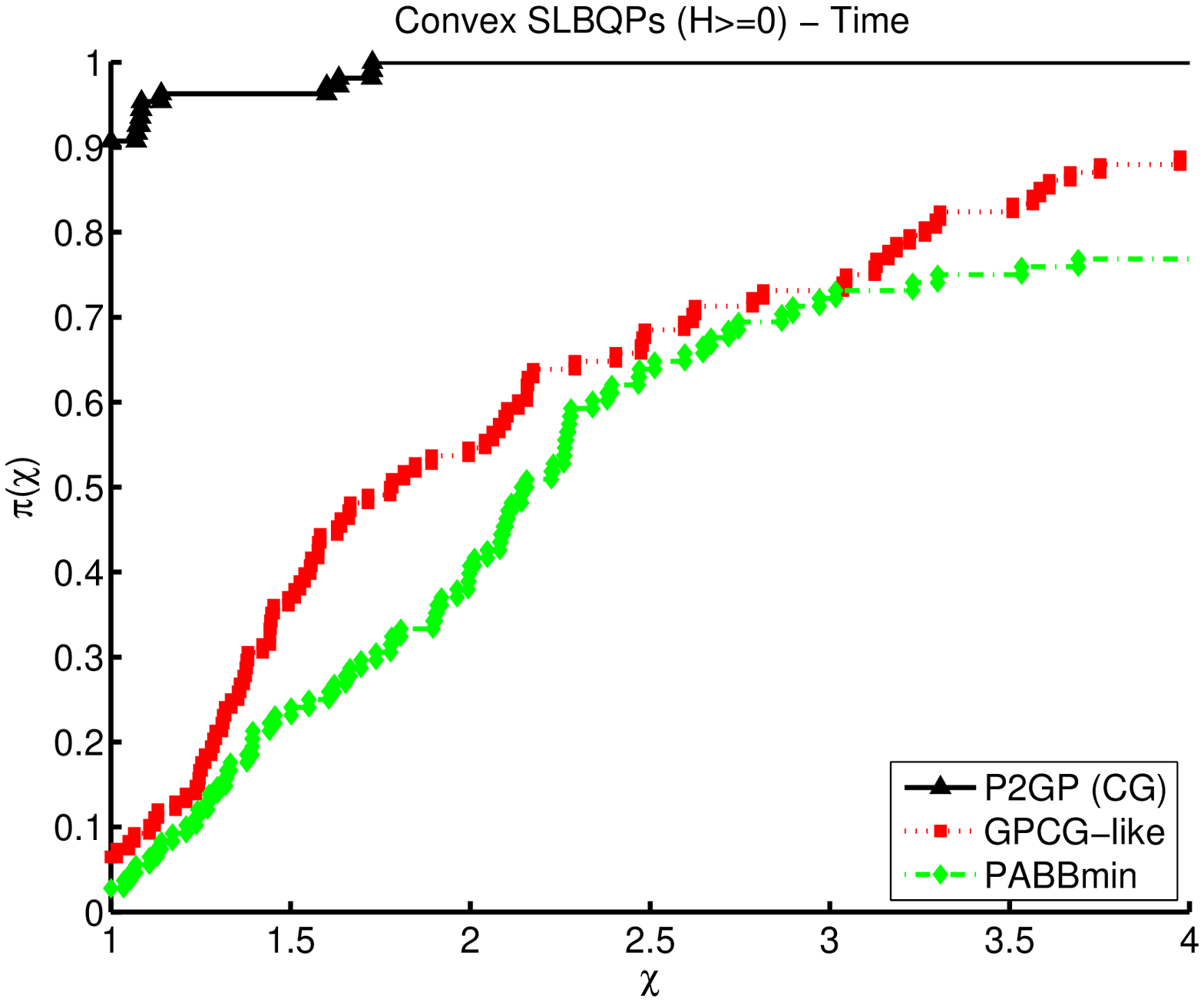} &
      \includegraphics[width=0.55\textwidth]{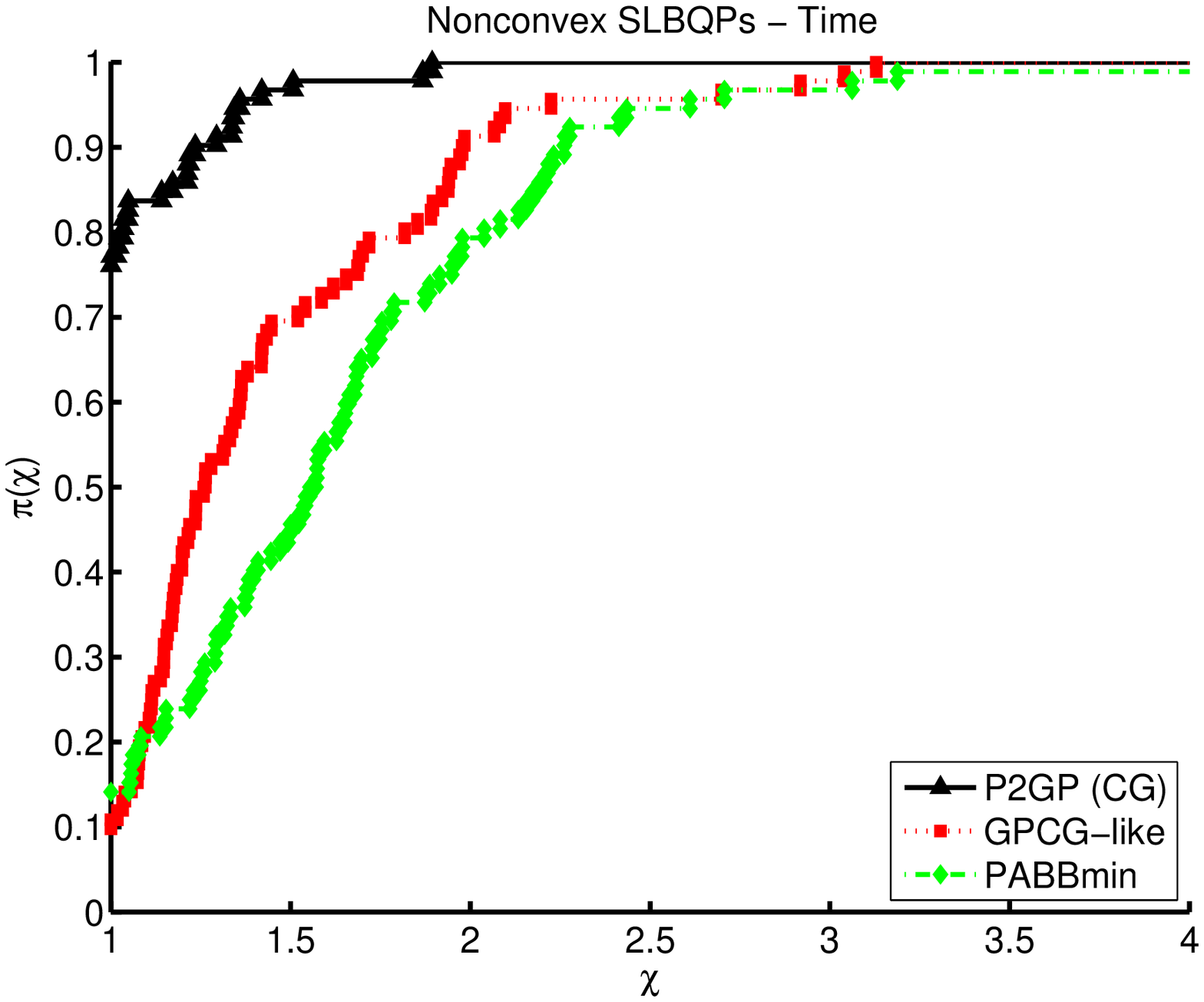}
   \end{tabular}
   \vskip -10pt
   \caption{Performance profiles (execution times) of P2GP with CG and SDC, 
PABB$_{\rm min}$, and GPCG-like on strictly convex SLBQPs with degenerate solutions
(top), convex SLBQPs (bottom left), nonconvex SLBQPs (bottom
right). \label{fig:deg_conv_nonconv_time}}
\end{figure}

Figure~\ref{fig:sconv_time} shows the performance profiles, $\pi(\chi)$,
of the three methods on the set of strictly convex SLBQPs with nondegenerate solutions,
using the execution time as performance metric. The profiles corresponding to
all the problems and to those with $\kappa (H) = 10^4$, $\kappa (H) = 10^5$,
and $\kappa (H) = 10^6$ are reported. We see that the version of P2GP using CG
in the minimization phase has by far the best performance. P2GP with SDC is faster than
the PABB$_{\rm min}$ and GPCG-like methods too. GPCG-like appears very sensitive to the condition
number of the Hessian matrix: its performance deteriorates as $\kappa (H)$ increases and the method
becomes less effective than PABB$_{\rm min}$ when $\kappa (H) = 10^6$.
This shows that the criterion used to terminate the minimization phase is more effective than the
criterion based on the bindingness of the active variables, especially as $\kappa (H)$
increases. We also report that the GPCG-like method has 6 failures over 36 runs for the
problems with $\kappa(H) = 10^6$.

For the previous problems, the performance profiles concerning the number
of matrix-vector products and the number of projections are also
shown, in Figure~\ref{fig:sconv_mv_p}. We see that PABB$_{\rm min}$ performs
the smallest number of matrix-vector products, followed by P2GP with GC,
and then by GPCG-like and P2GP with SDC. On the other hand, the number of
projections computed by P2GP with CG and with SDC is much smaller than
for the other methods; as expected, the maximum number of projections is
computed by PABB$_{\rm min}$.~This shows than the performance
of the methods cannot be measured only in terms of matrix-vector products;
the cost of the projections must also be considered, especially when
the structure of the Hessian makes the computational
cost of the matrix-vector products lower than $O(n^2)$. The good behavior
of P2GP results from the balance between matrix-vector products and projections. 

The performance profiles concerning the execution times on the strictly convex
SLBQPs with degenerate solutions, on the convex (but not strictly convex) SLBQPs,
and on the nonconvex ones are reported in Figure~\ref{fig:deg_conv_nonconv_time}.
Of course, the version of P2GP using the SDC solver was not applied to
the last two sets of problems. In the case of nonconvex problems, only 85\%
of the runs were considered, corresponding to the cases where the values
of the objective function at the solutions computed by the different methods
differ by less than 1\%. P2GP with CG is generally the best method, followed
by GPCG-like and then by PABB$_{\rm min}$. Furthermore, on strictly convex
problems with degenerate solutions, P2GP with SDC performs better than GPCG-like
and PABB$_{\rm min}$. GPCG-like is less robust than the other methods,
since it has 4 failures on the degenerate stricltly convex problems and 8 failures
on the convex ones.  This confirms the effectiveness of the proportionality-based
criterion.

For completeness, we also run the experiments on the strictly convex problems
with nondegenerate solutions by replacing the line search strategy in PABB$_{\rm min}$ with
a monotone line search along the feasible
direction~\cite[Section~2.3.1]{Bertsekas:1999}, which requires only one projection per
GP iteration.  We note that this line search does not guarantee in general that the sequence
generated by the GP method identifies in a finite number of steps the variables that are active
at the solution (see, e.g., \cite{DeAngelis:1993}). Nevertheless, we made experiments with the line search
along the feasible direction, to see if it may lead to any time gain in practice. The results obtained,
not reported here for the sake of space, show that the two line searches lead to comparable times
when the number of active variables at the solution is small, i.e., \texttt{naxsol} $= 0.1$. On the other hand, 
the execution time with the original line search is slightly smaller when the
number of active variables at the solution is larger.

Finally, the performance profiles concerning the execution times taken by the P2GP,
PABB$_{\rm min}$ and GPCG-like methods on the strictly convex BQPs with nondegenerate
and degenerate solutions, on the convex (but not strictly convex) BQPs, and on the nonconvex ones
are shown in Figure~\ref{fig:bqp}. Only 97\% of the runs on the nonconvex
problems are selected, using the same criterion applied to nonconvex SLBQPs.
P2GP with CG is again the most efficient method. The behavior of the methods is
similar to that shown on SLBQPs. However, P2GP with SDC and PABB$_{\rm min}$ have
closer behaviors, according to the smaller time required by projections
onto boxes, which leads to a reduction of the execution time of PABB$_{\rm min}$.
GPCG-like has again some failures: 6 on the strictly convex problems with nondegenerate
solutions, 5 on the ones with degenerate solutions, and 9 on the convex (but not
strictly convex) problems.

\begin{figure}[t]
   \centering
   \setlength{\tabcolsep}{-8pt}
   \begin{tabular}{cc}
      \hspace*{-9pt}
      \includegraphics[width=0.55\textwidth]{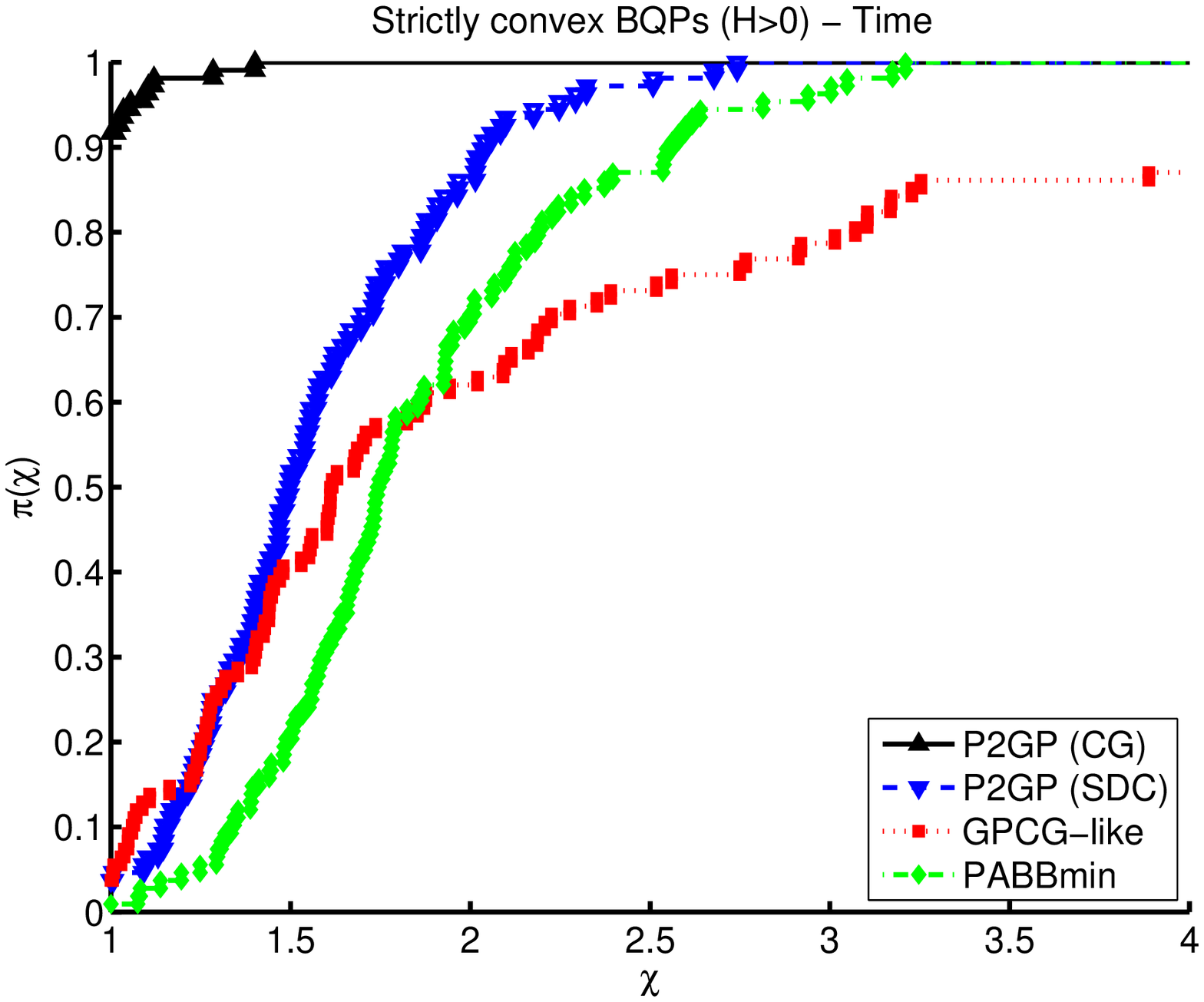} &
      \includegraphics[width=0.55\textwidth]{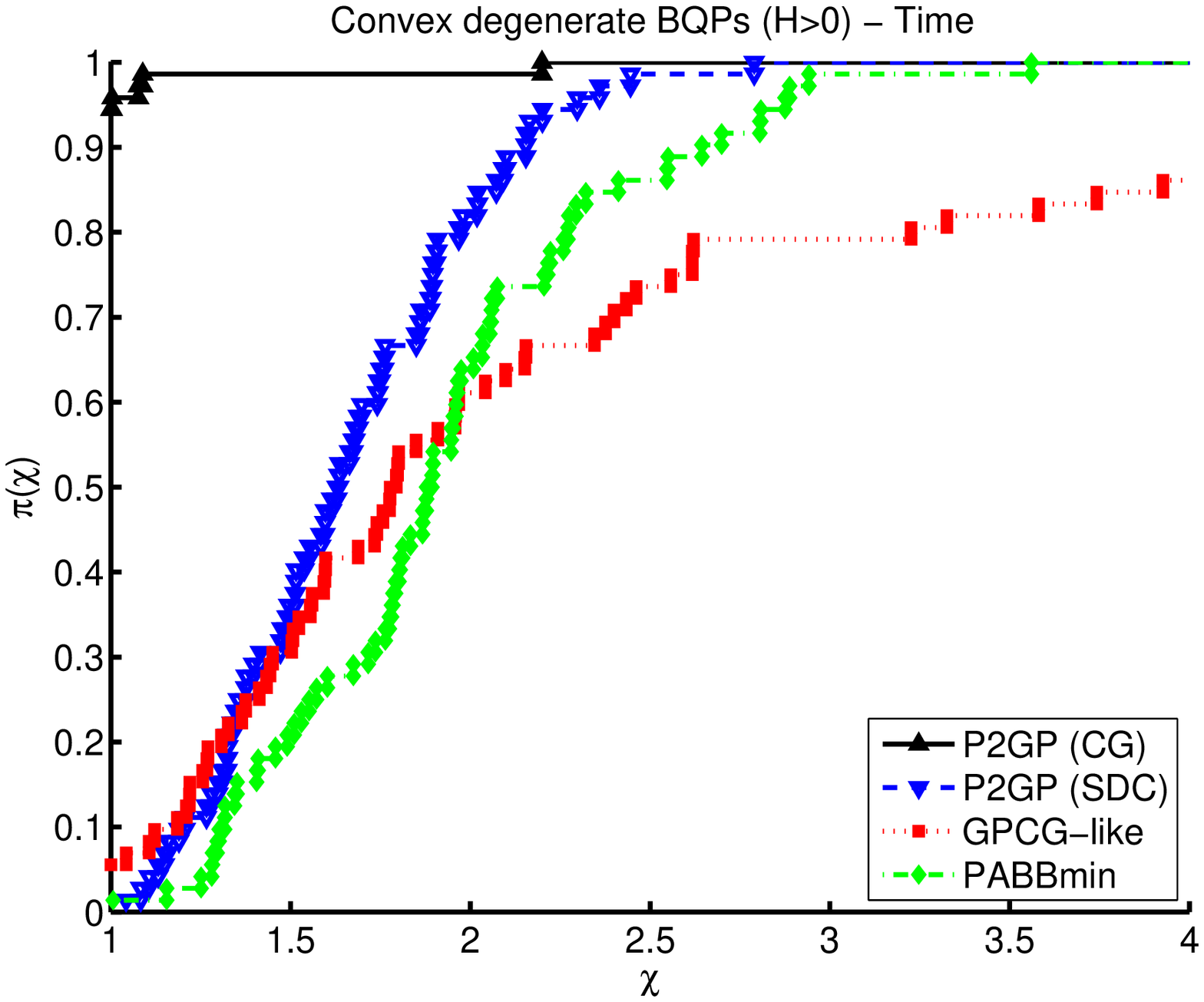}  \\[-4pt]
      \hspace*{-9pt}
      \includegraphics[width=0.55\textwidth]{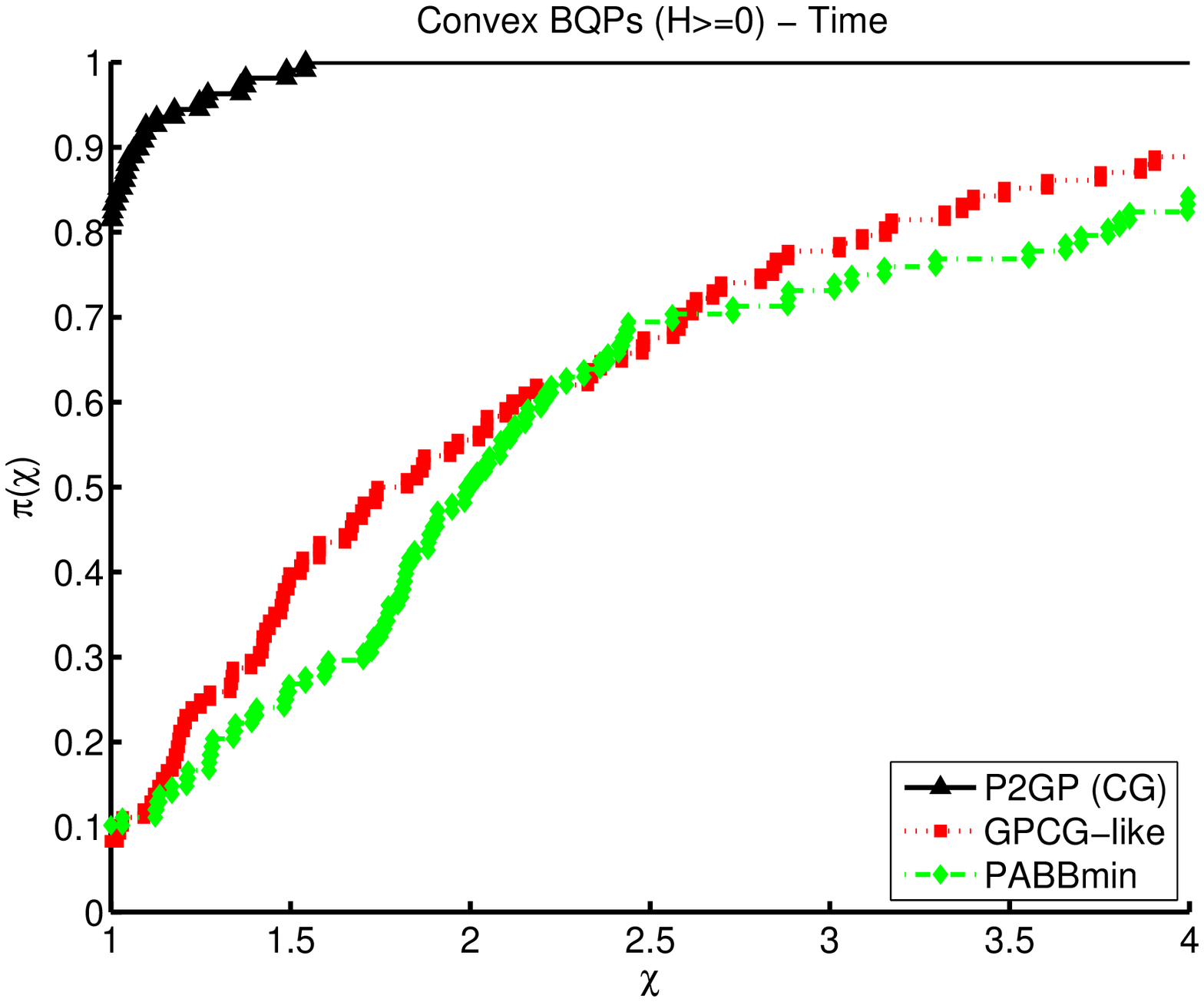} &
      \includegraphics[width=0.55\textwidth]{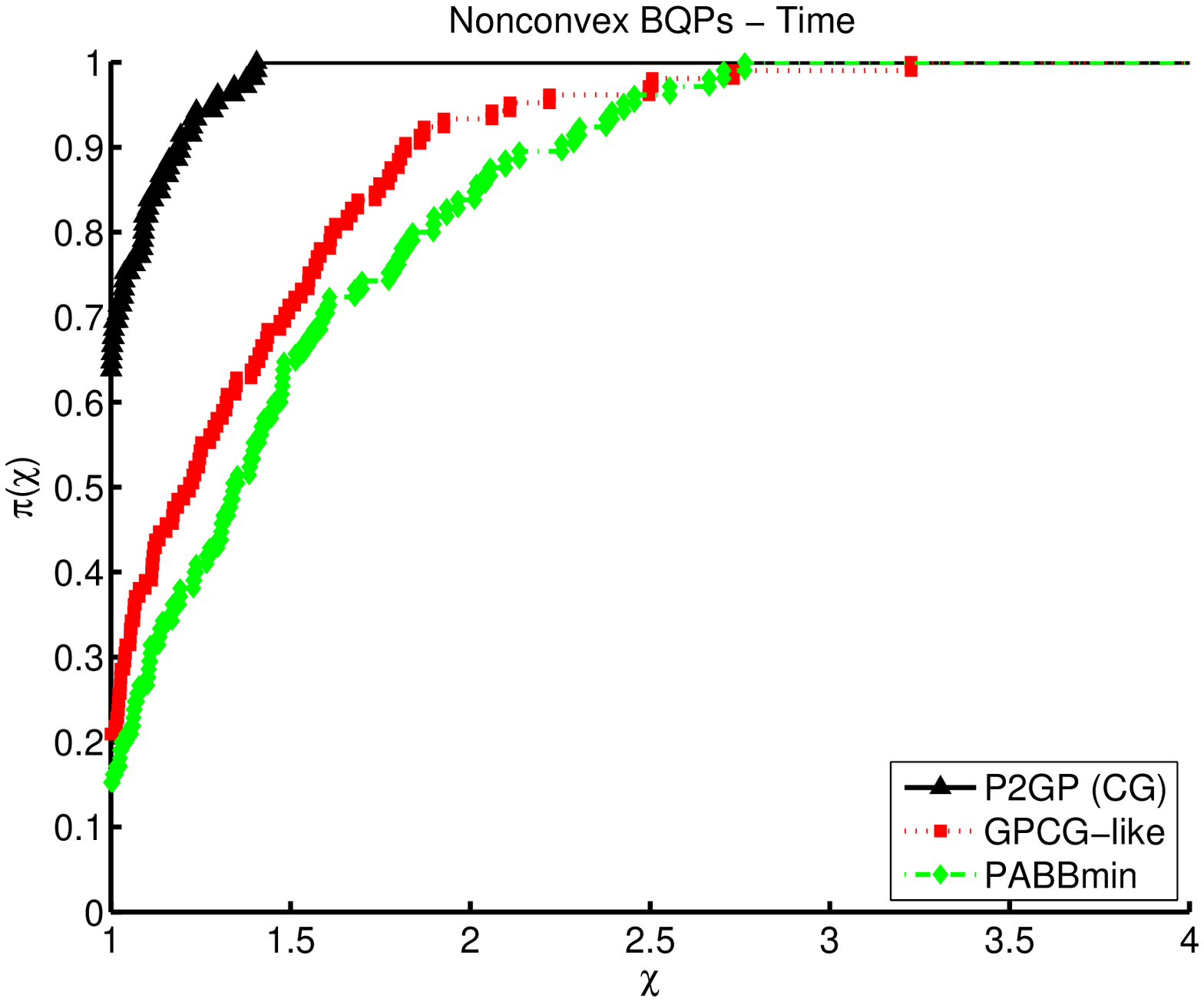}
   \end{tabular}
   \vskip -10pt
   \caption{Performance profiles (execution times) of P2GP with CG and SDC, 
PABB$_{\rm min}$, and GPCG-like on strictly convex BQPs with nondegenerate
solutions (top left), strictly convex BQPs with degenerate solutions (top right),
convex BQPs (bottom left), nonconvex BQPs (bottom right).\label{fig:bqp}}
\end{figure}

Now we compare P2GP (using CG) with BLG on the random problems. BLG was run in its
full-space mode (default mode), because the form of the Hessian \eqref{eq:randhess} does not allow to take
advantage of the subspace mode.
The stopping condition~\eqref{accur_req} was implemented in BLG, and
the code was run with the same tolerance and the same maximum numbers of matrix-vector products
and projections used for P2GP. Default values were used for the remaining parameters of BLG.
Of course, a comparison of the two codes in terms of execution time would be misleading, since BLG is written in C,
while P2GP has been implemented in Matlab. Therefore, we consider the matrix-vector products.
We do not show a comparison in terms of projections too, because BLG does a projection at each iteration,
and this generally results in many more projections than P2GP.
Performance profiles are provided in Figure~\ref{fig:blg_mv_p}.
The results concerning all the types of convex problems are shown together, since their
profiles are similar. On these problems P2GP appears more efficient than BLG; we also verified that the
objective function values at the solutions
computed by the two codes agree on at least six significant digits and are smaller for P2GP for 70\%
of the test cases. Furthermore, in four cases BLG does not satisfy condition~\eqref{accur_req} within the maximum
number of matrix-vector products and projections.
The situation is different for the nonconvex problems,
where the number of matrix-vector products performed by BLG is smaller. In this case, we verified that
BLG also used Frank-Wolfe directions, which were never chosen for the convex problems.
This not only reduced the number of matrix-vector products, but often led to smaller objective
function values. The values of the objective function at the solutions computed by the
two methods differ by less than 1\% for only 47\% of the test cases, which are the ones considered
in the performance profiles on the right of Figure~\ref{fig:blg_mv_p}. On the other hand, 
in three cases BLG performs the maximum number of matrix-vector products without achieving the required accuracy.

\begin{figure}[t]
   \centering
   \setlength{\tabcolsep}{-8pt}
   \begin{tabular}{cc}
      \hspace*{-9pt}
      \includegraphics[width=0.55\textwidth]{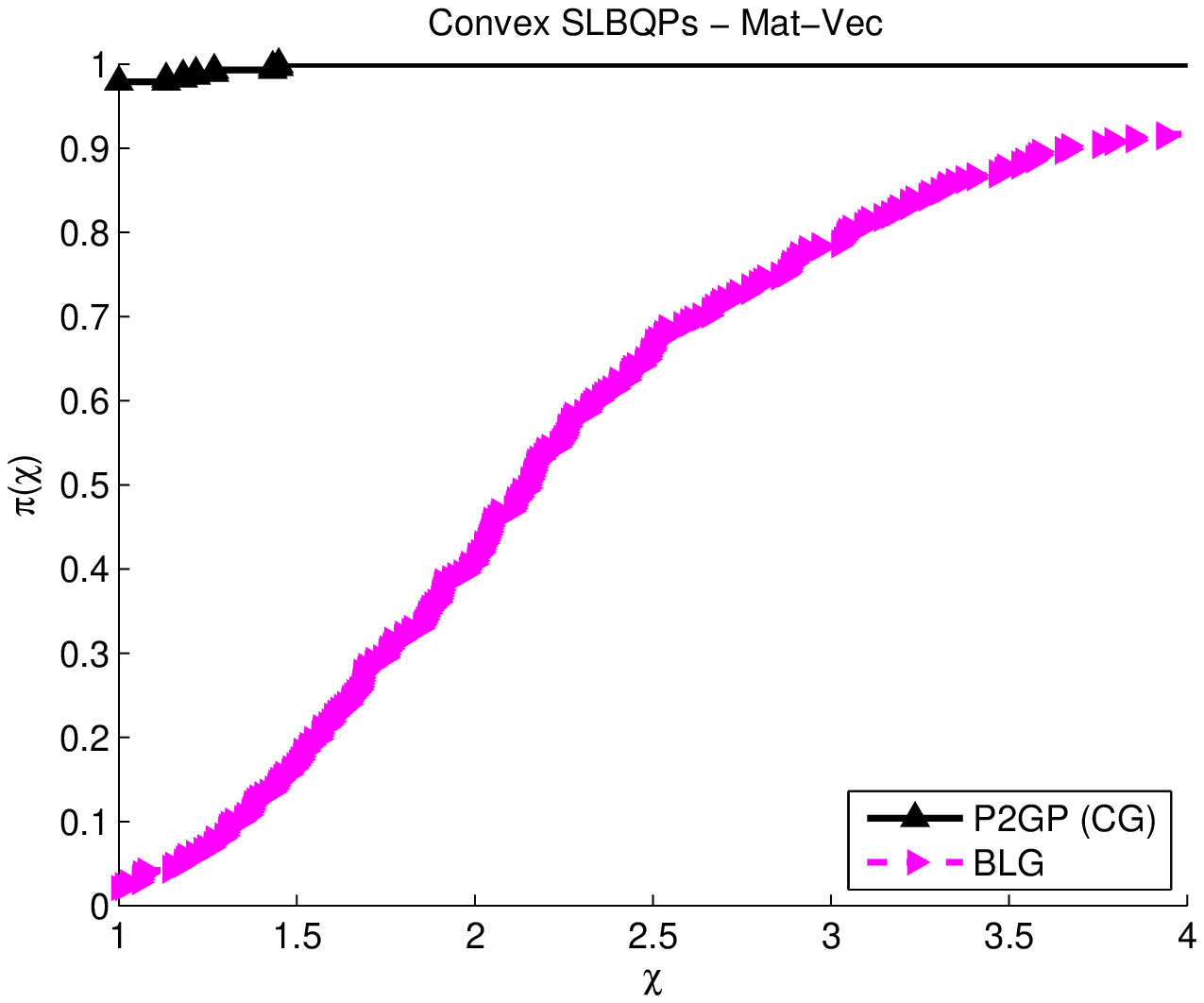} &
      \includegraphics[width=0.55\textwidth]{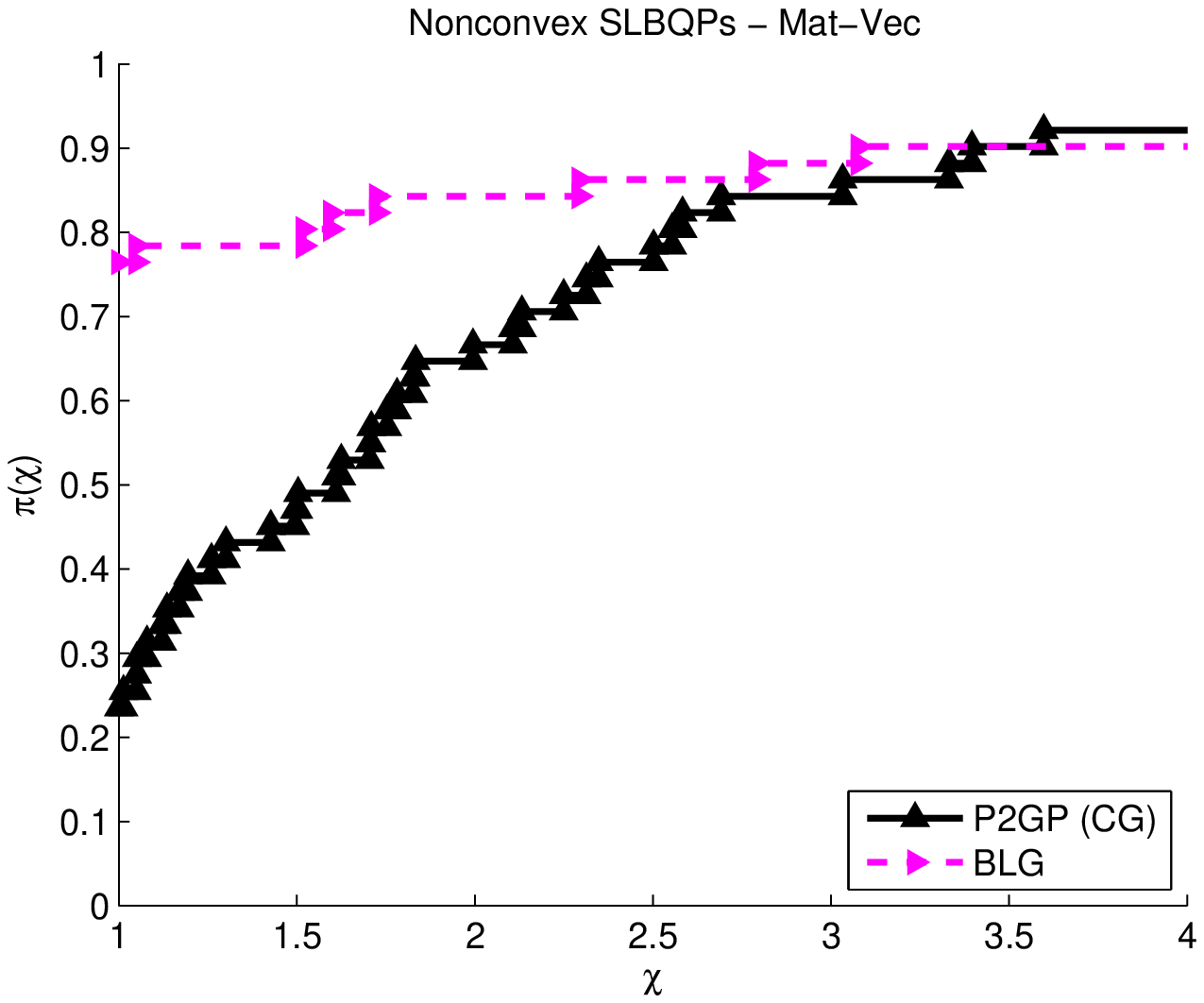}  \\[-4pt]
   \end{tabular}
   \vskip -10pt
   \caption{Performance profiles of P2GP, with CG, and BLG on convex (left) and nonconvex (right)
   SLBQPs: number of matrix-vector products.\label{fig:blg_mv_p}}
\end{figure}

\subsection{Results on SVM problems\label{sec:svmresults}}

In order to read the SVM problems, available in the LIBSVM format, BLG was run
through the SVMsubspace code, available from \url{http://users.clas.ufl.edu/hager/papers/Software/}.
Since we were interested in comparing P2GP with the GP implementation provided by BLG,
SVMsubspace was modified to have the SVM subspace equal to the entire space, i.e., to apply
BLG to the full SVM problem. For completeness we also run SVMsubspace in its subspace mode
(see~\cite{GonzalezLima:2011}), to see what the performance gain is with this feature. In the following,
we refer to the former implementation as BLGfull, and to the latter as SVMsubspace.

Following~\cite{GonzalezLima:2011}, BLGfull and SVMsubspace were used with their original stopping condition,
with tolerance $10^{-3}$. P2GP was terminated when the infinity norm of the projected gradient
was smaller then the same tolerance. With these stopping criteria, the two codes returned objective function values
agreeing on about six significant digits, with smaller function values generally obtained by P2GP.
At most 70000 matrix-vector products and 70000 projections were allowed, but they were never reached.

In Figure~\ref{fig:svmres}, left, the performance profiles (in logarithmic scale) concerning the matrix-vector
products of P2GP (with CG) and BLGfull are shown.
A comparison in terms of projections and execution times is not carried out for the same reasons explained
for the random problems. BLGfull appears superior than P2GP; on the other hand, we verified that
the number of projections performed by BLG is by far greater than that of P2GP for eight out of ten problems.
However, it must be noted that SVMsubspace is much faster than BLGfull, as shown by the performance
profiles concerning their execution times (see Figure~\ref{fig:svmres}, right). This confirms the great
advantage of performing reduced-size matrix-vector products in solving the subspace problems
for this class of test cases.

\begin{figure}[t]
   \centering
   \setlength{\tabcolsep}{-8pt}
   \begin{tabular}{cc}
      \hspace*{-9pt}
      \includegraphics[width=0.55\textwidth]{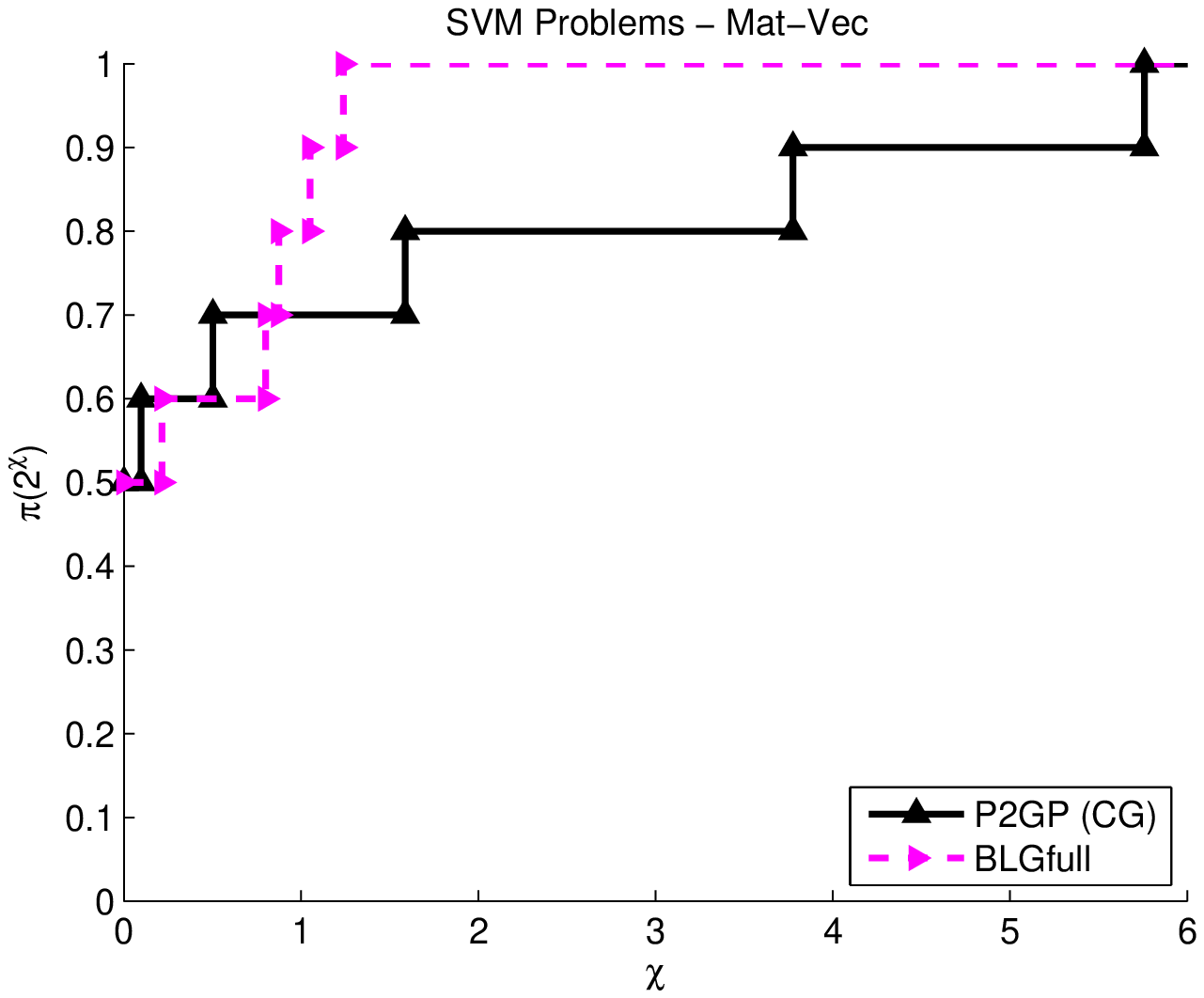} &
      \includegraphics[width=0.55\textwidth]{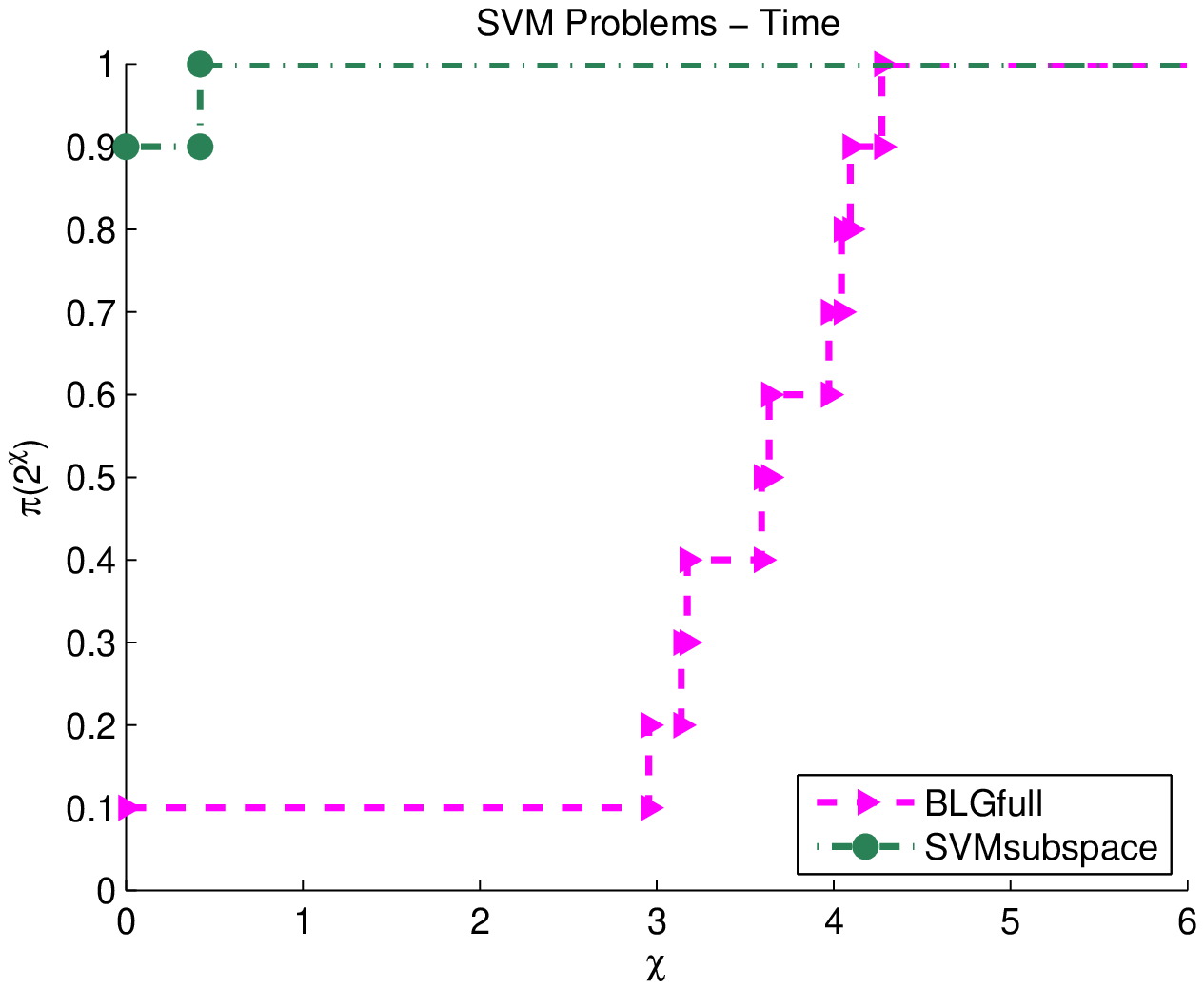}  \\[-4pt]
   \end{tabular}
   \vskip -10pt
   \caption{Performance profiles on SVM test problems: number of matrix-vector of P2GP, with CG, and BLG (left), and
   execution times of BLG and SVMsubspace (right).\label{fig:svmres}}
\end{figure}

\section{Concluding remarks\label{sec:conclusions}}

We presented P2GP, a new method for SLBQPs which has its roots in the GPCG method.
The most distinguishing feature of P2GP with respect to GPCG stands in the 
criterion used to stop the minimization phase. This is a critical issue,
since requiring high accuracy in this phase can be a useless and time-consuming task
when the face where a solution lies is far from being identified. 

Our numerical tests show a strong improvement of the computational performance when
the proportionality criterion is used to control the termination of the minimization phase.
In particular, the comparison of P2GP with an extension of GPCG to SLBQPs shows
the clear superiority of P2GP and its smaller sensitivity to the Hessian condition number.
Thus, proportionality allows to handle the minimization phase in a more clever way.
The numerical results also show that P2GP requires much fewer projections than
efficient GP methods like PABB$_{\rm min}$ and the one implemented in BLG.
This leads to a significant time saving,
especially when the Hessian matrix is sparse or has a structure that allows the computation
of the matrix-vector product with a computational cost smaller than $O(n^2)$, where $n$
is the size of the problem.
From the theoretical point of view, a nice consequence of using the proportionality criterion
is that finite convergence for strictly convex problems can be proved even in case of
degeneracy at the solution.

An interesting feature of P2GP is that it provides a general framework, allowing different
steplength rules in the GP steps, and different methods
in the minimization phase. The encouraging theoretical and computational results
suggest that this framework deserves to be further investigated, and possibly extended
to more general problems. For example, it would be interesting to extend P2GP to general
differentiable objective functions or to problems with bounds and
a few linear constraints.

The Matlab code implementing P2GP used in the experiments is available from
\url{https://github.com/diserafi/P2GP}. It includes
the test problem generator described in Section~\ref{sec:synthpbs}.

\vskip 10pt
\noindent
\textbf{Acknowledgments.}
We wish to thank William Hager for helpful discussions about the use of the BLG code
and for insightful comments on our manuscript.
We also express our thanks to the anonymous referees for their useful remarks and suggestions, which
allowed us to improve the quality of this work.

\bibliographystyle{siam}
\bibliography{biblio_grad}
	
\end{document}